\documentclass[letterpaper]{amsart}
\usepackage{amsmath}
\usepackage{fullpage}
\usepackage{mathtools}
\usepackage{amsthm}
\usepackage{amssymb}
\usepackage{amscd}
\usepackage{booktabs}
\usepackage[hyperfootnotes=false, colorlinks, linkcolor={blue},
citecolor={magenta}, filecolor={blue}, urlcolor={blue}]{hyperref}
\usepackage{resizegather}

\allowdisplaybreaks
\newcommand{\Q}{{\mathbb Q}}

\newcommand{\Z}{{\mathbb Z}}
\newcommand{\R}{{\mathbb R}}
\newcommand{\C}{{\mathbb C}}

\newcommand{\p}{\mathfrak p}
\newcommand{\OF}{{\mathfrak o}}
\newcommand{\GL}{{\rm GL}}

\newcommand{\GSp}{{\rm GSp}}

\newcommand{\Mat}{{\rm M}}

\newcommand{\diag}{{\rm diag}}

\newcommand{\Hom}{{\rm Hom}}
\newcommand{\SSp}{{\rm Sp}}

\newcommand*{\transp}[2][-3mu]{\ensuremath{\mskip1mu\prescript{\smash{\mathrm t\mkern#1}}{}{\mathstrut#2}}}
\newcommand{\Mod}[1]{\ (\textnormal{mod}\ #1)}
\newtheorem{lemma}{Lemma}[section]
\newtheorem{theorem}[lemma]{Theorem}
\newtheorem{corollary}[lemma]{Corollary}
\newtheorem{proposition}[lemma]{Proposition}

\begin{document}
\date{\today}
\title{The paramodular Hecke algebra}
\author{Jennifer Johnson-Leung}
\address{Department of Mathematics and Statistical Science\\University of Idaho\\Moscow, ID USA}
\email{jenfns@uidaho.edu}
\author{Joshua Parker}
\address{Division of Mathematics and Natural Science\\Elmira College\\Elmira, NY USA}
\email{jparker@elmira.edu}
\author{Brooks Roberts}
\address{Department of Mathematics and Statistical Science\\University of Idaho\\Moscow, ID USA}
\email{brooksr@uidaho.edu}

\begin{abstract}
We give a presentation via generators and relations of the local graded paramodular Hecke algebra of prime level. In particular, we prove that the paramodular Hecke algebra is isomorphic to the quotient of the free $\Z$-algebra generated by four non-commuting variables by an ideal generated by seven relations. Using this description, we derive rationality results at the level of characters and give a characterization of the center of the Hecke algebra. Underlying our results are explicit formulas for the product of any generator with any double coset.
\end{abstract}
\maketitle

\section{Introduction}
In this work we investigate the paramodular Hecke algebra $\mathcal{H}$ of level $\p$ over a fixed non-archimedean local field $F$ of characteristic zero with ring of integers $\OF$, maximal ideal $\p=(\varpi)\subset\OF$, and $q = \# \OF /\p$. We define $\mathcal{H}$ as the $\Z$-algebra spanned by the double cosets in $K\backslash\Delta/ K$, where $K$ is the paramodular group of level $\p$ and $\Delta$ is a certain semigroup in $\GSp(4,F)$, defined below in  \eqref{delta}; $\mathcal{H}$ has a natural grading determined by $\Delta$. Our main structural result is that $\mathcal{H}$ is isomorphic to the quotient of the (non-commutative) graded, free $\Z$-algebra generated by two elements in degree 1 and two elements in degree 2 by the two-sided ideal generated by seven explicit homogeneous relations. As a consequence of this description, we prove that $\mathcal{H}$ is non-commutative and has zero-divisors. We also show that the center of $\mathcal{H}$ is generated by one element in degree 1 and two elements in degree 2 which are algebraically independent.  Key ingredients for our analysis are the formulas of Theorems \ref{degree1multiplicationthm} and \ref{degree2multiplicationthm} which explicitly describe the product of a generator with any double coset. We note that a similar Hecke algebra studied in \cite{GK} is a quotient of the algebra under consideration in this paper.

In addition, we investigate the formal power series in the indeterminate $t$
\begin{equation}\label{formalpowerserieseq}
\sum_{k=0}^\infty T(q^k)t^k,
\end{equation}
where $T(q^k)$ is the sum of all double cosets of degree $k$. Series of this form for Hecke algebras are important in the theory of modular forms and have been studied by various authors (for example, \cite{Andrianov_1969,S2,TatTak1983,Bocherer1986}). In many cases, when the relevant Hecke algebra is commutative, such series have been shown to be rational functions with denominators related to Euler factors of appropriate $L$-functions. This question has also been investigated in some non-commutative settings (see \cite{Dulinski1999,Hyodo2015,hyodo2022formal}). We prove that, for certain characters $\chi$ of the Hecke algebra,  the formal power series
\begin{equation*}
\sum_{k=0}^\infty \chi(T(q^k))t^k
\end{equation*}
is a rational function in $t$. For those characters $\chi$ that arise in the theory of paramodular Siegel modular forms, the denominator is given by the Euler factor of the  spin $L$-function. Interestingly, our result seems to suggest that \eqref{formalpowerserieseq} is not a rational function in $\mathcal{H}(t)$: see the remarks after the proof of Theorem \ref{charrationalitytheorem}.

We expect that the explicit results of this paper will be useful for computational investigations of paramodular Siegel modular forms. In addition, our detailed results about $\mathcal{H}$ may provide a useful example for the general study of the Hecke algebras corresponding to parahoric subgroups of other reductive algebraic groups defined over a local field. Some of the results in this work appear in the PhD thesis of the second author. The first and second author were partially supported by the Renfrew Faculty Fellowship from the College of Science at the University of Idaho.

\section{Definitions and Notation}
\label{defnotationsec}

Throughout this paper $F$ is a fixed nonarchimedean local
field of characteristic zero.
Let $\OF$ be the ring of integers of $F$, let $\p$
be the prime ideal of $\OF$,  fix a generator $\varpi$
of $\p$, and let $q$ be the order of $\OF/\p$. If $x \in F^\times$,
then we let $v(x)$ be the unique integer such that
there exists $u \in \OF^\times$ so that $x=u\varpi^{v(x)}$.
Let 
\begin{equation}
J=
\begin{bsmallmatrix}
&&1&\\
&&&1\\
-1&&&\\
&-1&&
\end{bsmallmatrix}.
\end{equation}
We let $G=\GSp(4,F)$ be the group of $g \in \GL(4,F)$ for
which there exists $\lambda \in F^\times$ such that $\transp{g} J g = 
\lambda J$; such a $\lambda$ is uniquely determined and will be
denoted by $\lambda (g)$. We regard $G$ as group of td-type \cite{Cart}.
We denote the center of $G$ by $Z$; this consists of the elements
$\mathrm{diag}(z,z,z,z)$ for $z \in F^\times$.
We let $G_1=\SSp(4,F)$ be the subgroup of $g \in G$ with $\lambda(g)=1$.
Let 
\begin{equation}\label{delta}
\Delta =  G \cap
\begin{bsmallmatrix}
\OF & \OF & \p^{-1} & \OF\\
\p & \OF & \OF & \OF\\
\p & \p & \OF & \p \\
\p & \OF & \OF & \OF
\end{bsmallmatrix}.
\end{equation}
Then $\Delta$ contains $1$ and is a semi-group contained in $G$.
If $g \in \Delta$, then a calculation shows that $\det(g) \in \OF$
so that $\lambda(g) \in \OF$ and thus $v(\lambda(g)) \geq 0$.
We define
\begin{equation}
\label{Kdefeq}
K = \{ k \in \Delta: \lambda(k) \in \OF^\times\} = \{ k \in \Delta: v(\lambda(k)) = 0\}.
\end{equation}
Then $K$ is a maximal compact open subgroup of $G$, and we refer
to $K$ as the \emph{paramodular subgroup} of $G$ of level $\p$.
The group $K$ contains the Iwahori subgroup $I$ (see \eqref{iwahorieq})
and is thus a standard parahoric subgroup of $G$.
In total, there are seven standard parahoric subgroups of $G$:
$I$, $\mathrm{Kl}(\p)$, $\Gamma_0(\p)$, $u_1 \mathrm{Kl}(\p)u^{-1}$,
$\GSp(4,\OF)$, $u_1 \GSp(4,\OF) u_1^{-1}$, and $K$. Here, $u_1$
is as in \eqref{u1def} and $\mathrm{Kl}(\p)$ and $\Gamma_0(\p)$
are the Klingen and Siegel congruence subgroups of level $\p$, respectively;
for more information see, for example, Chap.~9 of \cite{JLRS}.

We let $\mathcal{H}$ be the ring $\mathcal{R}(K,\Delta)$
as defined in Sect.~3.1, p.~54 of \cite{S2} (note that the commensurator of $K$
in $G$ is $G$ as $K$ is compact and open). As a set, $\mathcal{H}$
consists of the formal sums $\sum_{C \in K \backslash \Delta / K} j(C) C$
with $j(C) \in \Z$ and $j(C)=0$ for all but finitely many $C \in K \backslash \Delta / K$.
The addition for $\mathcal{H}$ is defined in the obvious fashion, and
the double coset $K\cdot 1 \cdot K =K$ is the multiplicative identity of $\mathcal{H}$. The ring $\mathcal{H}$
is naturally a $\Z$-algebra, and we refer to $\mathcal{H}$ as the
\emph{paramodular Hecke algebra}.
If $C_1,C_2 \in K\backslash \Delta / K$, then we let
\begin{equation}
\label{c1c2prodeq}
C_1 \cdot C_2 = \sum\limits_{C \in K\backslash \Delta / K} m(C_1,C_2;C) C
\end{equation}
where $m(C_1,C_2;C)$ is a certain non-negative integer that is zero for all but finitely many $C \in K\backslash \Delta / K$.
Let $C_1,C_2,C \in K\backslash \Delta / K$. 
If $g_1,g_2,g \in \Delta$ are such that $C_1=Kg_1K$, $C_2=Kg_2K$, and $C=KgK$, 
then 
\begin{equation}
\label{gg1g2eq}
m(C_1,C_2;C) \neq 0 \iff g \in Kg_1Kg_2K
\end{equation}
and 
$m(C_1,C_2;C)$ is the number of distinct right $K$ cosets $Kh$ in $Kg_1^{-1}Kg 
\cap K g_2K$ (Sect.~3.1, p.~52 of \cite{S2}); this is also the number of distinct
left $K$ cosets $hK$ in $gKg_2^{-1}K \cap K g_1 K$. Using this, it is easy to
see that if $Kg_1 K = \sqcup_{i \in I} h_i K$ is a disjoint decomposition, then 
\begin{equation}
\label{doublemulteq}
m(C_1,C_2;C) = \#\{i \in I: h_i^{-1} g \in K g_2 K\},
\end{equation}
or equivalently,
\begin{equation}
\label{doublemulteq2}
m(C_1,C_2;C) = \#\{i \in I:
\text{there exists $k \in K$ such that $h_i^{-1} g k g_2^{-1} \in K$} \}.
\end{equation}
For $k \in \Z_{\geq 0}$ let $\mathcal{H}_k$ be the $\Z$-linear span in
$\mathcal{H}$ of the double cosets $KgK$ with $g \in \Delta$ and  $v(\lambda(g))=k$. Then
\begin{equation}
\mathcal{H} = \oplus_{k=0}^\infty \mathcal{H}_k, \qquad
\mathcal{H}_k \cdot \mathcal{H}_j \subset \mathcal{H}_{k+j}, \quad k,j \in
\Z_{\geq 0}
\end{equation}
where the second statement follows from \eqref{gg1g2eq}.
Thus, $\mathcal{H}$ is a graded $\Z$-algebra.
Let
\begin{align}
D_\Delta & =  \{ (a,b,c) \in \Z^3: 0 \leq a,\ 0 \leq b,\ 0 \leq c-a,\ 0 \leq c-b \}, \\
S_\Delta &= \{ (a,b,c) \in D_\Delta:  a \leq c-a,\  b \leq c-b\}.
\end{align}
The set
$S_\Delta$ is a monoid under addition. For $a, b, c \in \Z$ and $g \in \Delta$
we define
\begin{equation}
d(a,b,c) =
\mathrm{diag}(\varpi^a,\varpi^b,\varpi^{c-a},\varpi^{c-b}), \qquad
T(g) = KgK \in
\mathcal{H}.
\end{equation}
For $a,b,c \in \Z$ we have $d(a,b,c) \in \Delta$ if and only if $(a,b,c) \in D_{\Delta}$;
if $(a,b,c) \in D_\Delta$, then we  abbreviate
\begin{equation}
T(a,b,c) = T(d(a,b,c)),
\end{equation}
and we have $T(a,b,c) \in \mathcal{H}_c$.
Let
\begin{equation}
w=
\begin{bsmallmatrix}
&1&&\\
\varpi&&&\\
&&&\varpi\\
&&1&
\end{bsmallmatrix}.
\end{equation}
Then $w \in \Delta$, $w \Delta w^{-1} = \Delta$, $w K w^{-1} = K$, and $\lambda(w) = \varpi$.

The algebra $\mathcal{H}$ is related to a convolution algebra. Let
$\mathcal{H}'$ be the $\C$ vector space of all functions $f:G \to \C$ with compact support such that $f(k_1gk_2)=f(g)$ for $k_1,k_2 \in K$, $g \in G$, and such that the support of $f$ is contained in
$\Delta$. For $f_1,f_2 \in \mathcal{H}'$ let $f_1 * f_2 \in \mathcal{H}'$
be defined by $(f_1 * f_2)(g) = \int_G f_1(gh^{-1})f_2(h)\, dh$
for $g \in G$; here we use the Haar measure on $G$ such that $K$ has
measure $1$. Then $\mathcal{H}'$ is a $\C$-algebra with  product $*$.
There is an isomorphism of $\C$-algebras
\begin{equation}
\label{Hprimeisoeq}
\C \otimes_\Z \mathcal{H} \stackrel{\sim}{\longrightarrow} \mathcal{H}'
\end{equation}
that sends $a \otimes KgK$ to $a f_{KgK}$, where $a \in \C$, $g \in \Delta$,
and $f_{KgK}$ is the characteristic function of $KgK$.

\section{The structure of \texorpdfstring{$\mathcal{H}$}{H}}
\label{structuresec}

In this section we prove several results about the algebraic structure of $\mathcal{H}$.
After determining in Lemma \ref{BTdecomplemma} convenient representatives for
the double cosets that span $\mathcal{H}$, in Theorem \ref{degree1multiplicationthm} and Theorem \ref{degree2multiplicationthm} we explicitly compute the product of $X=T(0,0,1)$
and $Y_1 = T(0,1,2)$ with any double coset. Using these formulas we prove that
$\mathcal{H}$ is generated by $V=T(w)$, $X$, $Y_1$, and $Y_2=T(1,0,2)$, and we
give a presentation of $\mathcal{H}$ in terms of generators and relations.
We conclude this section by making some basic observations about $\mathcal{H}$,
and we relate $\mathcal{H}$ to an algebra studied in \cite{GK}.

\begin{lemma}
\label{BTdecomplemma}
If $C \in K \backslash \Delta / K$, then there exist $(a,b,c) \in S_\Delta$ and
$\delta \in \{0,1\}$ such that $C=Kw^\delta d(a,b,c) K$. If $(a,b,c),(a',b',c') 
\in S_\Delta$ and $\delta,\delta' \in \{0,1\}$ are such that $K w^{\delta} d(a,b,c)
K=Kw^{\delta'} d(a',b',c') K$, then $\delta=\delta'$ and $(a,b,c)=(a',b',c')$.
\end{lemma}
\begin{proof}
Let $C \in K \backslash \Delta / K$ and let $g \in \Delta$ be such that $C=KgK$.
Let $I$ be the standard Iwahori subgroup of $G$ so that 
\begin{equation}
\label{iwahorieq}
I = \GSp(4,\OF) 
\cap 
\begin{bsmallmatrix}
\OF &\OF &\OF &\OF\\
\p &\OF &\OF &\OF\\
\p &\p &\OF & \p\\
\p & \p & \OF &\OF
\end{bsmallmatrix}.
\end{equation}
Also, let 
\begin{equation}
s_0 = 
\begin{bsmallmatrix}
&&-\varpi^{-1}&\\
&1&&\\
\varpi&&&\\
&&&1
\end{bsmallmatrix}, \qquad
s_1=
\begin{bsmallmatrix}
&1&&\\
1&&&\\
&&&1\\
&&1&
\end{bsmallmatrix}, \qquad
s_2=
\begin{bsmallmatrix}
1&&&\\
&&&1\\
&&1&\\
&-1&&&
\end{bsmallmatrix}
\end{equation}
and 
\begin{equation}
\label{u1def}
u_1 =
\begin{bsmallmatrix}
&&& 1\\
&&-1&\\
&-\varpi&&\\
\varpi &&&
\end{bsmallmatrix}.
\end{equation}
Then $s_0, s_2 \in K$, and
\begin{equation}
\label{wu1releq}
w=u_1 k \quad \text{with}\quad k = -s_2 s_0 = -s_0 s_2 \in K.
\end{equation}
Let $T$ be the diagonal subgroup of $G$, let $T^\circ$ be the subgroup of $T$ 
of elements with diagonal elements in $\OF^\times$, and let $\mathrm{N}(T)$ be 
the normalizer of $T$ in $G$. The Iwahori-Weyl group (or extended affine Weyl 
group) $W^e$ is $N(T)/T^\circ$, and the images of $s_0,s_1,s_2,$ and $u_1$ 
generate $W^e$. By the Bruhat-Tits decomposition (see for example
\cite{Cart} or chap.~9 of \cite{JLRS}) there exists a product
$r=r_1\cdots r_t$, where $r_1,\dots,r_t \in \{s_0^{\pm 1},s_1^{\pm 
1},s_2^{\pm 1},u_1^{\pm 1} \}$, such that $IgI=IrI$. Hence, $C=KgK=KrK$. 
If $s \in \{s_0,s_1,s_2,u_1\}$, then $s^{-1} = sd(a,b,c)k$ for some $(a,b,c)
\in \Z^3$ and $k \in T^\circ=T\cap K$. Hence, $C = K h_1\cdots h_\ell d(a,b,c) K$ for
some $h_1,\dots,h_\ell \in \{s_0,s_1,s_2,u_1\}$ and $(a,b,c) \in \Z^3$.
If 
$h_1=s_0$ or $s_2$, then $h_1 \in K$ and $C=Kh_2\cdots h_\ell d(a,b,c) K$. If
$h_1=s_1$, then $C=Kwh_2\cdots h_\ell d(a,b,c)K$ for some $(a,b,c) \in \Z^3$ as 
$s_1=wd(-1,0,-1)$. If $h_1=u_1$, then $C=Kwh_2\cdots h_\ell K$ as $u_1=kw$ for 
some $k \in K$. In every case we have $C=Kw^\delta h_2\cdots h_\ell 
d(a,b,c)K$ for some $\delta \in \Z_{\geq 0}$ and $(a,b,c) \in \Z^3$. 
Applying the same reasoning successively to $h_2,\dots, h_\ell$, and using 
that $w$ normalizes $K$, we conclude that $C=K w^\delta d(a,b,c) K$ for some 
$\delta \in \Z_{\geq 0}$ and $(a,b,c) \in \Z^3$. Since $w^2=d(1,1,2)$ we may
assume that $\delta \in \{0,1\}$. Also, since $C \subset \Delta$ we have $w^\delta
d(a,b,c) \in \Delta$; this implies that $d(a,b,c) \in \Delta$ so that
$(a,b,c) \in D_\Delta$. As
$C = K w^\delta d(a,b,c) K =w^\delta K d(a,b,c) K$, $d(a,b,c) \in \Delta$, and $s_0,s_2 \in K$,
we may conjugate $d(a,b,c)$ by $s_0$ or $s_2$, if necessary, to arrange
that $(a,b,c) \in S_\Delta$.
For the uniqueness assertion, assume that $K w^\delta d(a,b,c) K = K 
w^{\delta'} d(a',b',c')K$ for some $\delta,\delta' \in \{0,1\}$ and 
$(a,b,c),(a',b',c') \in S_\Delta$. First assume that $\delta=\delta'=0$. Then
$d(a,b,c)kd(-a',-b'-c')=k'$ for some $k,k' \in K$. Applying the similitude 
homomorphism $\lambda$ and taking valuations we get $c=c'$. Next, using that 
\begin{equation}\label{Kform}
k=\begin{bsmallmatrix}
a_1 & a_2 & b_1 \varpi^{-1} & b_2\\
a_3 \varpi & a_4 & b_3 & b_4\\
c_1 \varpi & c_2 \varpi & d_1 & d_2 \varpi \\
c_3 \varpi & c_4 & d_3 & d_4 
\end{bsmallmatrix} \in K
\implies
\begin{bsmallmatrix}
a_1 & b_1 \\
c_1 & d_1 
\end{bsmallmatrix},
\begin{bsmallmatrix}
a_4 & b_4 \\
c_4 & d_4
\end{bsmallmatrix} \in \GL(2,\OF),
\end{equation}
a calculation and argument shows that $a=a'$ and $b=b'$. A similar analysis shows that
$\delta \neq \delta'$ is impossible. Finally, if $\delta=\delta'=1$, then 
$K  d(a,b,c) K = K 
 d(a',b',c')K$ as $w$ normalizes $K$; hence $a=a'$, $b=b'$, and $c=c'$ by the 
first case.
\end{proof}

It is worth noting that, as used in the proof of Lemma \ref{BTdecomplemma},
\begin{equation}
Kd(a,b,c)K = Kd(c-a,b,c)K=Kd(a,c-b,c)K=Kd(c-a,c-b,c)K
\end{equation}
for $a,b,c \in \Z$ (since $s_0,s_2 \in K$).

Using Lemma~\ref{BTdecomplemma} we find that for $k \in \Z_{\geq 0}$,
\begin{equation}
\label{rankeq}
\mathrm{rank}_\Z (\mathcal{H}_k )= (2k^2+4k+3+(-1)^k)/4,
\end{equation}
and  $\mathrm{rank}_\Z (\mathcal{H}_{k+1}) = \mathrm{rank}_\Z (\mathcal{H}_k) +k
+1 +(1+(-1)^{k+1})/2$. The first  few ranks are
$$
\begin{array}{cccccccccc}
\toprule
k & 0&1&2&3&4&5&6&7&8\\
\midrule
\mathrm{rank}_\Z (\mathcal{H}_k ) & 1 &2 &5& 8 & 13 & 18 & 25 & 32  & 41 \\
\bottomrule
\end{array}
$$

We define
\begin{align}
V & = T(w) \in \mathcal{H}_1, \label{Vdef}\\
X &= T(\mathrm{diag}(1,1,\varpi,\varpi)) = T(0,0,1) \in \mathcal{H}_1, \label{Xdef}\\
Y_1 &= T(\mathrm{diag}(1,\varpi,\varpi^2,\varpi))=T(0,1,2) \in \mathcal{H}_2,\label{Y1def} \\
Y_2 &= T(\mathrm{diag}(\varpi,1,\varpi,\varpi^2)) = T(1,0,2) \in \mathcal{H}_2.\label{Y2def}
\end{align}

For the following lemma see Proposition~3.7 of \cite{S2}.
\begin{lemma}
\label{normalizerlemma}
If $g,h \in \Delta$ and $gK=Kg$, then $T(gh) = T(g) T(h)$ and $T(hg) = T(h)T(g)$ in $\mathcal{H}$.
\end{lemma}
By Lemma \ref{normalizerlemma}, since $w$ normalizes $K$,  if $(a,b,c)\in D_\Delta$,
then
\begin{equation}
\label{Vouteq}
T(w d(a,b,c) ) = T(w) T(a,b,c) = V T(a,b,c).
\end{equation}
Similarly,
\begin{gather}
V^2 = Kw^2 K  = K
\begin{bsmallmatrix}
\varpi&&&\\
&\varpi&&\\
&&\varpi&\\
&&&\varpi
\end{bsmallmatrix}K = T(1,1,2), \label{V2eq}\\
V^2T(a,b,c)=T(a+1, b+1, c+2) \label{inductionfacteq}
\end{gather}
for $(a,b,c)\in D_\Delta$.
We see that $V^2$ is in the center $Z(\mathcal{H})$ of $\mathcal{H}$.

If $\gamma: G \to G$ is an automorphism (anti-automorphism) such that 
$\gamma(\Delta)=\Delta$ and $\gamma(K)=K$, then the map $\mathcal{H} \to 
\mathcal{H}$ determined by $KgK \mapsto K \gamma (g) K$ for $g \in \Delta$ is a 
$\Z$-algebra automorphism (anti-automorphism).
These facts can be verified using the explicit descriptions of
the product for $\mathcal{H}$ from section~\ref{defnotationsec}.
Let $\alpha:\mathcal{H} \to
\mathcal{H}$ and $\beta:\mathcal{H} \to \mathcal{H}$ denote the automorphism 
and anti-automorphism determined by conjugation by $w$ and the 
anti-automorphism of $G$ that sends $\begin{bsmallmatrix} a&b \\ c&d 
\end{bsmallmatrix}$ to $\begin{bsmallmatrix} \transp{d} & -\transp{b} \\ 
-\transp{c} & \transp{a} \end{bsmallmatrix}$, respectively. Then 
$\alpha^2=\beta^2=\mathrm{Id}$,
\begin{equation}
\label{alphabetapropeq}
\alpha \left( T(a,b,c)\right)=T(b,a,c), \qquad \beta \left( T(a,b,c)\right)
=T(c-a,c-b,c) = T(a,b,c)
\end{equation}
for $(a,b,c)\in D_\Delta$, and
\begin{gather}
\alpha( V )=V, \qquad \alpha (X)
=X, \qquad \alpha (Y_1)=Y_2, \qquad \alpha(Y_2)=Y_1,\label{alphageneq}\\
\beta(V) = V, \qquad \beta(X) = X, \qquad \beta(Y_1) = Y_1, \qquad
\beta(Y_2)=Y_2. \label{betageneq}
\end{gather}

For $g=\begin{bsmallmatrix} a&b \\ c&d \end{bsmallmatrix}  \in \GL(2,\OF)$, $t \in \OF^\times$, and $x,y,z \in F$ we let
\begin{equation}
\label{kgdef}
k(t,g) =
\begin{bsmallmatrix}
t&&&\\
&a&&b\\
&&\det(g)t^{-1}&\\
&c&&d
\end{bsmallmatrix}
\end{equation}
and
\begin{equation}
\label{uelldef}
u(x,y,z)=
\begin{bsmallmatrix}
1&x&z&y\\
&1&y&\\
&&1&\\
&&-x&1
\end{bsmallmatrix},
\qquad
\ell(x,y,z)
=
\begin{bsmallmatrix}
1&&&\\
x&1&&\\
z&y&1&-x\\
y&&&1
\end{bsmallmatrix}.
\end{equation}
We note that the elements in \eqref{kgdef} and \eqref{uelldef}
are in $K$ and $G_1$, respectively.

\begin{lemma}
\label{divlemma}
Let $d_1,d_2,d_3,d_4,c_1,c_3 \in \Z_{\geq 0}$ and assume that $d_1+d_3=d_2+d_4$
and $c_1+c_3=1$. Let $g \in \GL(2,\OF)$ and $t \in \OF^\times$. Then
\begin{align}
&K \diag(\varpi^{c_1},1,\varpi^{c_3},\varpi)
k(t,g) \diag(\varpi^{d_1},\varpi^{d_2},\varpi^{d_3},\varpi^{d_4}) K \nonumber \\
&\qquad = K \diag(\varpi^{\min(c_1+d_1,c_3+d_3)},\varpi^{q_1},
\varpi^{\max(c_1+d_1,c_3+d_3)},\varpi^{q_2}) K \label{divlemmaeq1}
\end{align}
where
\begin{equation}
\label{divlemmaeq2}
(q_1,q_2)
\in
\begin{cases}
\{ (d_2,d_4+1),(d_2+1,d_4)\} &\text{if $d_2 \leq d_4-1$},\\
\{(d_2,d_2+1)=(d_4,d_4+1)\} & \text{if $d_2=d_4$},\\
\{(d_4,d_2+1),(d_4+1,d_2)\} & \text{if $ d_4+1 \leq d_2$}.
\end{cases}
\end{equation}
\end{lemma}
\begin{proof}
If $M=\begin{bsmallmatrix} m_1&m_2 \\ m_3&m_4 \end{bsmallmatrix} \in \Mat(2,\OF) \cap \GL(2,F)$,
then there exist unique $g_1,g_2 \in \GL(2,\OF)$ and $e_1,e_2 \in \Z_{\geq 0}$ such that
$g_1M g_2 = \begin{bsmallmatrix} \varpi^{e_1} & \\ & \varpi^{e_2} \end{bsmallmatrix}$
and $e_1 \leq e_2$; moreover, $\varpi^{e_1}$ is a generator of the ideal
$(m_1,m_2,m_3,m_4)$ of $\OF$, and $\varpi^{e_2}$ is a generator of the ideal
$(\det(M) / \varpi^{e_1})$ of $\OF$ (see for example chap.~II, sec.~15 (p.~26) of \cite{Ne}).
Using this one may prove that there exist $h,h' \in \GL(2,\OF)$ such that
\begin{equation}
\label{divlemmaeq3}
h\begin{bsmallmatrix} 1\vphantom{\varpi^{d_2}}& \\ &\varpi\vphantom{\varpi^{d_2}} \end{bsmallmatrix}
g
\begin{bsmallmatrix} \varpi^{d_2} & \\ & \varpi^{d_4} \end{bsmallmatrix} h'
= \begin{bsmallmatrix} \varpi^{q_1} \vphantom{\varpi^{d_2}}& \\ & \varpi^{q_2} \vphantom{\varpi^{d_2}}\end{bsmallmatrix}
\end{equation}
with $q_1 \leq q_2$ and
\begin{equation}
\{q_1,q_2\} = \{ d_2,d_4+1\} \quad \text{or} \quad
\{q_1,q_2\} = \{d_2+1,d_4\}.
\end{equation}
We now have
\begin{align}
&K \diag(\varpi^{c_1},1,\varpi^{c_3},\varpi)
k(t,g) \diag(\varpi^{d_1},\varpi^{d_2},\varpi^{d_3},\varpi^{d_4}) K \nonumber \\
&\qquad = K \diag(\varpi^{c_1},1,\varpi^{c_3},\varpi)
k(1,g) \diag(\varpi^{d_1},\varpi^{d_2},\varpi^{d_3},\varpi^{d_4}) K \nonumber \\
&\qquad = K k(1,h) \diag(\varpi^{c_1},1,\varpi^{c_3},\varpi)
k(1,g) \diag(\varpi^{d_1},\varpi^{d_2},\varpi^{d_3},\varpi^{d_4})k(1,h')  K \nonumber \\
&\qquad = K \diag(\varpi^{\min(c_1+d_1,c_3+d_3)},\varpi^{q_1},
\varpi^{\max(c_1+d_1,c_3+d_3)},\varpi^{q_2})K. \nonumber
\end{align}
This completes the proof.
\end{proof}

\begin{lemma}
\label{gg1g2lemma}
Let $k \in K$ and $(a,b,c) \in S_\Delta$.  There exists $(\delta,(e,f,g)) \in
\{0,1\} \times \Z^3$ such that
\begin{gather}
K d(0,0,1) k d(a,b,c) K = K w^\delta d(e,f,g)K,\label{gg1g2lemmaeq1}\\
\begin{array}{l}
(\delta,(e,f,g))
\in
\{ (0,(a,b,c+1)),\ (0,(a,b+1,c+1)),\ (0,(a+1,b,c+1)),   \\
\hspace{2.3cm}  (0,(a+1,b+1,c+1)),\ (1,(a,b,c))\},\label{gg1g2lemmaeq2}
\end{array}  \\
(\delta,(e,f,g)) \in  \{0,1\} \times S_\Delta. \label{gg1g2lemmaeq100}
\end{gather}
\end{lemma}
\begin{proof}
We first assume that $b\geq a$. Assume that $2a=c$. Then since
$2a \leq 2b \leq c$ we have $2a=2b=c$ so that $d(a,b,c)=
\mathrm{diag}(a,a,a,a)$. Therefore,
$$
K d(0,0,1)kd(a,b,c)K = K d(a,a,2a+1) K = Kd(a,b,c+1)K.
$$
This verifies the lemma in this case. Assume that $2a<c$.
We will first prove that there exists $(\delta,(e,f,g)) \in
\{0,1\} \times \Z^3$ such that \eqref{gg1g2lemmaeq1} and \eqref{gg1g2lemmaeq2}
hold.
Let
\begin{equation}
\label{gg1g2lemmaeq2p5}
h=d(0,0,1) k d(a,b,c).
\end{equation}
By Lemma 3.3.1 of \cite{RS} there is a disjoint decomposition
\begin{equation}
\label{gg1g2lemmaeq3}
K = \mathrm{Kl}(\p) s_0 \sqcup \bigsqcup_{v \in \OF/\p} \mathrm{Kl}(\p) u(0,0,v \varpi^{-1}).
\end{equation}
Here, $\mathrm{Kl}(\p)$ is the Klingen congruence subgroup
$$
\mathrm{Kl}(\p) = \GSp(4,\OF) \cap
\begin{bsmallmatrix}
\OF & \OF & \OF & \OF\\
\p & \OF & \OF & \OF\\
\p & \p & \OF & \p \\
\p & \OF & \OF & \OF
\end{bsmallmatrix}.
$$
Assume first that $k$ is in the union on the right of \eqref{gg1g2lemmaeq3}.
Since  $\mathrm{Kl}(\p)$ has an
Iwahori decomposition there exist $x_1,y_1,z_1,x_2,y_2,z_2 \in \OF$, $t \in \OF^\times$,
and $g \in \GL(2,\OF)$
such that
\begin{equation}
\label{gg1g2lemmaeq4}
k = \ell(x_1\varpi,y_1\varpi,z_1\varpi) k(t,g) u(x_2,y_2,z_2\varpi^{-1}).
\end{equation}
Hence,
\begin{align}
&KhK
= K d(0,0,1)  \ell(x_1\varpi,y_1\varpi,z_1\varpi) k(t,g) u(x_2,y_2,z_2\varpi^{-1})
d(a,b,c) K \nonumber \\
&\qquad= K   \ell(x_1\varpi,y_1\varpi^2,z_1\varpi^2) d(0,0,1)k(t,g) d(a,b,c)\nonumber \\
&\qquad \quad \times u(x_2\varpi^{b-a},y_2\varpi^{c-a-b},z_2\varpi^{c-2a-1})  K\nonumber \\
&\qquad= K   d(0,0,1)k(t,g) d(a,b,c)  K. \label{gg1g2lemmaeq5}
\end{align}
For the last equality we used that $b-a \geq 0$, $c-a-b\geq 0$, and $c-2a-1 \geq -1$.
Applying  Lemma \ref{divlemma} to \eqref{gg1g2lemmaeq5} we conclude that
$$
KhK = K d(a,b,c+1)K \quad \text{or}\quad KhK = Kd(a,b+1,c+1)K.
$$
This verifies that there exists $(\delta,(e,f,g)) \in
\{0,1\} \times \Z^3$ such that \eqref{gg1g2lemmaeq1} and \eqref{gg1g2lemmaeq2}
hold in this case. Now assume that $k$
is in $\mathrm{Kl}(\p) s_0$. Let $k = k's_0$ where $k' \in \mathrm{Kl}(\p)$. Since $s_0
\in K$ we have
\begin{equation}
\label{gg1g2lemmaeq6}
K h K = K d(0,0,1)k' d(c-a,b,c) K.
\end{equation}
Since  $\mathrm{Kl}(\p)$ has an
Iwahori decomposition there exist $x_1,y_1,z_1,x_2,y_2,z_2 \in \OF$, $t \in \OF^\times$,
and $g \in \GL(2,\OF)$ such that
\begin{equation}
\label{gg1g2lemmaeq7}
k' = \ell(x_1\varpi,y_1\varpi,z_1\varpi) u(x_2,0,0)u(0,y_2,z_2) k(t,g).
\end{equation}
By \eqref{gg1g2lemmaeq6} and \eqref{gg1g2lemmaeq7},
\begin{align}
&K h K = K d(0,0,1) \ell(x_1\varpi,y_1\varpi,z_1\varpi) u(x_2,0,0)u(0,y_2,z_2) k(t,g) d(c-a,b,c) K\nonumber \\
&\qquad = K  \ell(x_1\varpi,y_1\varpi^2,z_1\varpi^2) d(0,0,1) u(x_2,0,0)u(0,y_2,z_2) k(t,g) d(c-a,b,c) K\nonumber \\
&\qquad = K   d(0,0,1) u(x_2,0,0)u(0,y_2,z_2) k(t,g) d(c-a,b,c) K\nonumber \\
&\qquad = K    u(x_2,0,0)d(0,0,1)u(0,y_2,z_2) k(t,g) d(c-a,b,c) K\nonumber \\
&\qquad = K   u(0,y_2\varpi^{-1},z_2\varpi^{-1})d(0,0,1) k(t,g) d(c-a,b,c) K\nonumber \\
&\qquad = K   u(0,y_2\varpi^{-1},0)d(0,0,1) k(t,g) d(c-a,b,c) K. \label{gg1g2lemmaeq8}
\end{align}
Assume that $y_2 \in \p$. Then from \eqref{gg1g2lemmaeq8} we have
\begin{equation}
\label{gg1g2lemmaeq9}
KhK= K   d(0,0,1) k(t,g) d(c-a,b,c) K.
\end{equation}
Applying  Lemma \ref{divlemma} to \eqref{gg1g2lemmaeq9} we conclude that
$$
KhK = K d(a+1,b,c+1)K \quad \text{or}\quad KhK = Kd(a+1,b+1,c+1)K.
$$
This verifies that there exists $(\delta,(e,f,g)) \in
\{0,1\} \times \Z^3$ such that \eqref{gg1g2lemmaeq1} and \eqref{gg1g2lemmaeq2}
hold in this case. Assume that $y_2 \in \OF^\times$.
Using that
\begin{align}
\label{gg1g2lemmaeq95}
u(0,y,0)=\ell(0,y^{-1},0)ws_0s_2 \mathrm{diag}(-(y\varpi )^{-1},(y\varpi)^{-1},-y,y)
\ell(0,y^{-1},0),
\end{align}
which holds for all $y \in F^\times$,
and that $y_2 \in \OF^\times$ and $s_0,s_2 \in K$, we find that by \eqref{gg1g2lemmaeq8}
\begin{align}
&K h K
= K \ell(0,y_2^{-1}\varpi,0)ws_0s_2\mathrm{diag}(-y_2^{-1},y_2^{-1},-y_2\varpi^{-1},y_2\varpi^{-1})
\ell(0,y_2^{-1}\varpi,0)\nonumber \\
&\qquad \quad \times d(0,0,1) k(t,g) d(c-a,b,c) K \nonumber \\
&\qquad = w K d(0,0,-1)
\ell(0,y_2^{-1}\varpi,0) d(0,0,1) k(t,g) d(c-a,b,c) K\nonumber \\
&\qquad = w K  \ell(0,y_2^{-1},0) k(t,g) d(c-a,b,c) K. \label{gg1g2lemmaeq10}
\end{align}
There exist $x_3,y_3,z_3 \in \OF$  such that
$$
\ell(0,y_2^{-1},0) k(t,g)  =
k(t,g) \ell(0,y_3,0)\ell(x_3,0,0)\ell(0,0,z_3).
$$
From \eqref{gg1g2lemmaeq10} we thus obtain
\begin{align}
KhK &=  w K  k(t,g) \ell(0,y_3,0)\ell(x_3,0,0)\ell(0,0,z_3) d(c-a,b,c) K \nonumber \\
&=  w K  d(c-a,b,c)\ell(0,y_3 \varpi^{b-a}, 0) \ell(x_3 \varpi^{c-a-b},0,0)
\ell(0,0, z_3 \varpi^{c-2a})   K \nonumber \\
&=  w K  d(c-a,b,c)\ell(0,y_3 \varpi^{b-a}, 0)  K. \label{gg1g2lemmaeq11}
\end{align}
For the last equality we used $c-a-b>0$  and $c-2a>0$. We have $y_3\varpi^{b-a} \in \OF$
since $b \geq a$. Assume that $y_3\varpi^{b-a} \in \p$. Then $\ell(0,y_3 \varpi^{b-a}, 0)
\in K$ and
\begin{align*}
K h K
& = w K d(c-a,b,c) K \\
& = w K s_0 d(c-a,b,c) s_0^{-1} K \\
& = w K d(a,b,c) K \\
& = K w d(a,b,c) K.
\end{align*}
This verifies that there exists $(\delta,(e,f,g)) \in
\{0,1\} \times \Z^3$ such that \eqref{gg1g2lemmaeq1} and \eqref{gg1g2lemmaeq2}
hold in this case. Assume that $y_3\varpi^{b-a} \in \OF^\times$.
Since $b \geq a$ we have $a=b$ so that $y_3\varpi^{b-a}=y_3$.
Using that
\begin{equation}
\label{gg1g2lemmaeq12}
\ell(0,y,0)
= u(0,y^{-1},0)
ws_0s_2 \mathrm{diag} (y \varpi^{-1}, -y \varpi^{-1}, y^{-1}, -y^{-1}) u(0,y^{-1},0),
\end{equation}
which holds for all $y \in F^\times$, we find from \eqref{gg1g2lemmaeq10} that:
\begin{align*}
K h K
&=  w K  d(c-a,b,c) u(0,y_3^{-1},0)
ws_0s_2 \\
&\quad \times \mathrm{diag} (y_3 \varpi^{-1}, -y_3 \varpi^{-1}, y_3^{-1}, -y_3^{-1}) u(0,y_3^{-1},0)  K\\
&=  w K  u(0,y_3^{-1},0) d(c-a,b,c)
ws_0s_2  \mathrm{diag} (\varpi^{-1}, \varpi^{-1}, 1, 1)   K\\
&=  w K d(c-a,b,c)
ws_0s_2  \mathrm{diag} (\varpi^{-1}, \varpi^{-1}, 1, 1)   K\\
&=  w K ws_0s_2  d(c-b,a,c)\mathrm{diag} (\varpi^{-1}, \varpi^{-1}, 1, 1)   K\\
&=  w^2 K  d(c-b,a,c)\mathrm{diag} (\varpi^{-1}, \varpi^{-1}, 1, 1)   K\\
&=  \mathrm{diag}(\varpi,\varpi,\varpi,\varpi) K  d(c-b,a,c)\mathrm{diag} (\varpi^{-1}, \varpi^{-1}, 1, 1)   K\\
& = K \mathrm{diag}(\varpi^{c-b},\varpi^a,\varpi^{b+1},\varpi^{c-a+1}) K\\
& = K s_0\mathrm{diag}(\varpi^{c-b},\varpi^a,\varpi^{b+1},\varpi^{c-a+1}) s_0^{-1}K\\
& = K d(b+1,a,c+1) K\\
& = K d(a+1,b,c+1) K,
\end{align*}
where for the last step we used that $a=b$.
This completes the verification that there exists $(\delta,(e,f,g)) \in
\{0,1\} \times \Z^3$ such that \eqref{gg1g2lemmaeq1} and \eqref{gg1g2lemmaeq2}
hold in the case $b \geq a$. Still under the assumption that $b \geq a$,
we will now prove that there exists $(\delta,(e,f,g)) \in \{0,1\} \times \Z^3$
such that \eqref{gg1g2lemmaeq1}, \eqref{gg1g2lemmaeq2}, and \eqref{gg1g2lemmaeq100}
hold. We have already proven that there exists $(\delta,(e,f,g)) \in \{0,1\} \times \Z^3$
such that \eqref{gg1g2lemmaeq1} and \eqref{gg1g2lemmaeq2} hold. If \eqref{gg1g2lemmaeq100}
holds, then there is nothing more to prove; assume that $(\delta,(e,f,g)) \notin \{\pm 1\} \times S_\Delta$.
From the definition of $S_\Delta$ we see that necessarily $\delta =0$ and
$(e,f,g) \in \{(a,b+1,c+1),(a+1,b,c+1),(a+1,b+1,c+1)\}$. Assume that $(e,f,g)=
(a,b+1,c+1)$. Since $(e,f,g) \notin S_\Delta$ and $(a,b,c) \in S_\Delta$ we have
$b=c-b$. We now have
\begin{align*}
KhK & = K w^\delta d(e,f,g) K \\
& = K d(a,b+1,c+1) K \\
& = K s_2 d(a,b+1,c+1) s_2^{-1} K \\
& = K d(a,(c+1)-(b+1),c+1) K \\
& = K d(a,b,c+1)K.
\end{align*}
Let $(\delta',(e',f',g')) = (0,(a,b,c+1))$. Then \eqref{gg1g2lemmaeq1}, \eqref{gg1g2lemmaeq2}, and \eqref{gg1g2lemmaeq100} still
hold after replacing $(\delta,(e,f,g))$ by $(\delta',(e',f',g'))$.
Similar arguments deal with the cases $(e,f,g)=(a+1,b,c+1)$ and $(e,f,g) =
(a+1,b+1,c+1)$. This completes the proof of the lemma in the case $b \geq a$.

Now assume that $a>b$. Then $w^{-1}kw \in K$ and $(b,a,c) \in S_\Delta$; by the previous
case applied to $w^{-1}kw$ and $(b,a,c)$, there exists
$(\delta,(f,e,g)) \in \{0,1\} \times S_\Delta$ such that
\begin{gather}
K d(0,0,1) w^{-1}kw d(b,a,c) K = K w^\delta d(f,e,g)K,\label{gg1g2lemmaeqB1}\\
\begin{array}{l}
(\delta,(f,e,g))
\in
\{ (0,(b,a,c+1)),\ (0,(b,a+1,c+1)),\ (0,(b+1,a,c+1)),   \\
\hspace{2.3cm}  (0,(b+1,a+1,c+1)),\ (1,(b,a,c))\},\label{gg1g2lemmaeqB2}
\end{array}  \\
(\delta,(f,e,g)) \in  \{0,1\} \times S_\Delta. \label{gg1g2lemmaeqB100}
\end{gather}
Conjugating \eqref{gg1g2lemmaeqB1} by $w$, and using that $w$ normalizes $K$,
we obtain \eqref{gg1g2lemmaeq1}. Since \eqref{gg1g2lemmaeqB2} is equivalent
to \eqref{gg1g2lemmaeq2}, and since \eqref{gg1g2lemmaeqB100}
is equivalent to \eqref{gg1g2lemmaeq100}, the proof is complete.
\end{proof}

\begin{theorem} 
\label{degree1multiplicationthm}
Let $(a,b,c) \in S_\Delta$. There exist unique  integers
$n_1$, $n_2$, $n_3$, $n_4$, and $n_5$ such that
\begin{align}
&T(0,0,1)T(a,b,c)=
n_1 T(a,b,c+1)+n_2 T(a,b+1,c+1)\nonumber \\
&\qquad +n_3T(a+1,b,c+1)+n_4T(a+1,b+1,c+1)
+n_5 T(w)T(a,b,c) \label{degree1multiplicationthmeq1000}
\end{align}
and
\begin{equation}
\text{for $1 \leq i \leq 5$, $n_i  =0$ if $(e,f,g) \notin S_\Delta$,}
\end{equation}
where $(e,f,g)$ is the triple of integers in the $i$-th summand in
\eqref{degree1multiplicationthmeq1000}.
The integers $n_i$ are  as in the following table:
$$
\renewcommand{\arraystretch}{1}
\begin{array}{cllllll}
\toprule
\multicolumn{2}{c}{\mathrm{conditions}} & n_1 & n_2 & n_3 & n_4 & n_5 \\
\midrule
b<a & a=c-a  &1 & q^2&0 &0 &q^2-1\\
&a+1 = c-a &1 & q^2&q+1 &q^3+q^2 &q^2-1\\
&a+2 \leq c-a &1 &q^2 & q&q^3 &q^2-1\\
\midrule
b=a & b = c-b &1 &0 & 0&0 & 0\\
 & b+1 = c-b &1 &q+1 &q+1 & q^3+2q^2+q&q-1 \\
 & b+2 \leq c-b &1 &q &q &q^3 &q-1 \\
\midrule
a<b & b=c-b &1 &0 &q^2 & 0&q^2-1 \\
& b+1=c-b &1 &q+1 &q^2 &q^3+q^2 &q^2-1 \\
& b+2 \leq c-b &1 &q &q^2 & q^3&q^2-1 \\
\bottomrule
\end{array}
$$
\end{theorem}
\begin{proof} Let
\begin{align*}
A &= \left(\{0,1\} \times S_\Delta\right) \cap
\{ (0,(a,b,c+1)),\ (0,(a,b+1,c+1)),   \\
&\quad \ (0,(a+1,b,c+1)),\ (0,(a+1,b+1,c+1)),\ (1,(a,b,c))\}.
\end{align*}
By Lemma \ref{BTdecomplemma}, the double cosets
$K w^\delta d(e,f,g) K$ for $(\delta,(e,f,g)) \in A$
are mutually distinct.
By \eqref{c1c2prodeq} we have
\begin{equation}
\label{degree1multiplicationthmeq1}
T(0,0,1)T(a,b,c)
=
\sum\limits_{C \in K \backslash \Delta / K} m(C) C,
\end{equation}
where we abbreviate $m(C) =m(d(0,0,1),d(a,b,c); C)$.
Let $C \in K \backslash \Delta / K$ be
such that $m(C) \neq 0$, and let $g \in \Delta$
be such that $C=KgK$. By \eqref{gg1g2eq} we have $g \in K d(0,0,1)Kd(a,b,c)K$.
Hence, there exists $k \in K$ such that $g \in Kd(0,0,1)kd(a,b,c)K$, i.e.,
$C=Kd(0,0,1)kd(a,b,c)K$. By Lemma \ref{gg1g2lemma}, there exists
$(\delta,(e,f,g)) \in A$ such that $C = K w^\delta d(e,f,g) K$. It  follows that
\begin{align}
&T(0,0,1)T(a,b,c) =
\sum\limits_{(\delta, (e,f,g)) \in A} m\left( K w^\delta d(e,f,g)K \right)\cdot K w^\delta d(e,f,g)K.
\label{degree1multiplicationthmeq2}
\end{align}
This proves the existence
of the $n_i$ as in the statement of the theorem (note that we also use \eqref{Vouteq}).
The uniqueness of the $n_i$ follows from the fact that the
$Kw^\delta d(e,f,g) K$ for $(\delta,(a,b,c)) \in \{0,1\} \times S_\Delta$
form a  $\Z$-basis for $\mathcal{H}$ (recall the definition
of $\mathcal{H}$ and see Lemma \ref{BTdecomplemma}).

Next, we calculate the $n_i$.
We begin with some observations.
First, if $b=a$ and $b=c-b$, then $T(a,b,c)=
\mathrm{diag}(\varpi^a,\varpi^a,\varpi^a,\varpi^a)$
and $T(0,0,1)T(a,b,c)=T(a,b,c+1)$ by Lemma \ref{normalizerlemma};
thus, for this case, we have $n_1=1$ and $n_2=n_3=n_4=n_5=0$.
We may thus assume that $b \neq a$ or $b< c-b$.
Second, by the first paragraph, if $(e,f,g)$ is the triple
in the $i$-th summand in \eqref{degree1multiplicationthmeq1000}
and $(e,f,g) \notin S_\Delta$, then $n_i=0$: this accounts
for the remaining zero entries in the above table. Third,
by applying $\alpha$ and using \eqref{alphabetapropeq}
and \eqref{alphageneq}, the values for $n_3$ can be obtained
from those for $n_2$.

With these observations in place, we will explain how to calculate
the remaining $n_i$. To do this we will use \eqref{doublemulteq2}
with $g_1=d(0,0,1)$ and $g_2=d(a,b,c)$ and the left coset
decomposition $Kg_1K= \sqcup_{i \in I} h_i K$ from Lemma~6.1.2
of \cite{RS}. In this coset decomposition the representatives
are organized into four parameterized families as follows:
\begin{gather*}
r_1 = r_1(x,y,z)=\begin{bsmallmatrix}1&&z\varpi^{-1}&y\\&1&y&x\\&&1&\\&&&1\end{bsmallmatrix}d(1,1,1),\qquad
r_2= r_2(x,z)=\begin{bsmallmatrix}1&x&z\varpi^{-1}&\\&1&&\\&&1&\\&&-x&1\end{bsmallmatrix}d(1,0,1),\\
r_3=r_3(x,y)=s_0\begin{bsmallmatrix}1&&&y\\&1&y&x\\&&1&\\&&&1\end{bsmallmatrix}d(1,1,1),\qquad  r_4=r_4(x)=s_0\begin{bsmallmatrix}1&x&&\\&1&&\\&&1&\\&&-x&1\end{bsmallmatrix}d(1,0,1),
\end{gather*}
where $x,y,z\in\OF/\p$. Rather than recounting the details for the use
of \eqref{doublemulteq2} in every instance we will instead describe
two illustrative cases.

Calculation of $n_1$ when $b<a$ and $a+2\leq c-a$. Let $g = d(a,b,c+1)$.
Then arguments by contradiction, via explicit calculations, show that
there exist no $k \in K$ such that $r_j^{-1} gkg_1^{-1} \in K$
for $j=1,2,3,$ and $4$ except when $j=4$ and $r_4=r_4(0,0,0)$,
when we have $r_4^{-1} g \cdot 1 \cdot g_2^{-1}=-s_0 \in K$. Thus, $n_1=1$.

Calculation of $n_4$ when $b=a$ and $b+1=c-b$. Let $g=d(a+1,b+1,c+1)$.
For any $x,y,z \in \OF$ we have $r_1^{-1}gkg_2^{-1} =1 \in K$
where
$$
k =
\begin{bsmallmatrix}
1&&z\varpi^{-1} & y \\
&1&y&x\\
&&1&\\
&&&1
\end{bsmallmatrix} \in K;
$$
hence, by \eqref{doublemulteq2}, the $r_1$ family contributes $q^3$ to $n_4$.
Next, for any $x,z \in \OF$ we have $r_2^{-1} gkg_2^{-1} =s_2 \in K$ for
$$
k =
s_2
\begin{bsmallmatrix}
1&& z \varpi^{-1} & x \\
&1&x&\\
&&1&\\
&&&1
\end{bsmallmatrix} \in K;
$$
hence, the $r_2$ contributes $q^2$ to $n_4$. And for any $x,y \in \OF$
we have $r_3^{-1} g k g_2^{-1}=1 \in K$ for
$$
k =
s_0
\begin{bsmallmatrix}
1&&&y\\
&1&y&x\\
&&1&\\
&&&1
\end{bsmallmatrix} \in K;
$$
thus, the $r_3$ contributes $q^2$ to $n_4$. Finally, for any $x \in \OF$
we have $r_4^{-1} g k g_2^{-1} =s_2 \in K$ for
$$
k = s_0 s_2
\begin{bsmallmatrix}
1&&&x \\
&1&x&\\
&&1&\\
&&&1
\end{bsmallmatrix} \in K;
$$
hence, the $r_4$ family contributes $q$ to $n_4$. Adding the contributions, we
find that
$n_4=q^3+2q^2+q$.
\end{proof}

\begin{theorem} 
\label{degree2multiplicationthm}
Let $(a,b,c) \in S_\Delta$.
There exist unique  integers
$m_1$, $m_2$, $m_3$, $m_4$, $m_5$, and $m_6$ such that
\begin{align}
&T(0,1,2)T(a,b,c) = m_1 T(a,b+1,c+2)+m_2 T(a+1,b+1,c+2)\nonumber \\
&\qquad+ m_3 T(a+2,b+1,c+2) +m_4 T(w) T(a,b,c+1)\nonumber \\
&\qquad+ m_5 T(w) T(a+1,b,c+1) +m_6 T(w) T(a+1,b+1,c+1)\label{degree2multiplicationthmeq1}
\end{align}
and
\begin{equation}
\label{degree2multiplicationthmeq2}
\text{for $1 \leq i \leq 6$, $m_i  =0$ if $(e,f,g) \notin S_\Delta$,}
\end{equation}
where $(e,f,g)$ is the triple of integers occuring in the $i$-th summand of
\eqref{degree2multiplicationthmeq1}.
The  integers $m_i$  are  as in the following table:
\begin{gather*}
\renewcommand{\arraystretch}{1}
\begin{array}{clrrrrrr}
\toprule
\multicolumn{2}{c}{\mathrm{conditions}} &m_1&m_2&m_3&m_4&m_5&m_6\\
\midrule
b<a& a=c-a  & q&0 &0 &q-1 &0 &0\\
& a+1=c-a    &q &q^2 &0 &q-1 & q^2-1&0\\
&a+2 =c-a    &q & q^2-q&q^3+q^2 &q-1 &q^2-q &0\\
&a+3 \leq c-a    &q & q^2-q&q^3 &q-1 & q^2-q&0\\
\midrule
b=a& a=c-a  & 1& 0&0 &0 &0 &0\\
&a+1=c-a    & 1& q^2&0 &0 &q^2-1 &0\\
&a+2 =c-a    & 1&q^2-q &q^3+q^2 &0 &q^2-q &0\\
&a+3 \leq c-a    &1 &q^2-q &q^3 &0 &q^2-q &0\\
\midrule
a<b& b=c-b  &1 &q^3-q^2 &q^4+q^3 &0 &q^3-q^2 &0\\
\mathrm{and}&b+1=c-b    &1 &q^3-q^2 &q^4+q^3  &0 &q^3-q^2  &q^4-q^2\\
a+2=c-a&b+2 \leq c-b   & 1&q^3-q^2 & q^4+q^3 & 0&q^3-q^2  &q^4-q^3\\
\midrule
a<b& b=c-b &1 &q^3-q^2 &q^4 &0 &q^3-q^2  &0\\
\mathrm{and}&b+1=c-b   &1 &q^3-q^2 &q^4 &0 &q^3-q^2  &q^4-q^2\\
a+2<c-a&b+2 \leq c-b   & 1&q^3-q^2 &q^4 &0 &q^3-q^2  &q^4-q^3\\
\bottomrule
\end{array}
\end{gather*}
\end{theorem}
\begin{proof}
Assume first that $c=a+b$. Then since $0 \leq a \leq c-a$ and $0 \leq b \leq c-b$
we have $a=b=c-a=c-b$ so that $a=b$ and $c=2a=2b$. Evidently,
$d(a,b,c) = \mathrm{diag}(\varpi^a,\varpi^a,\varpi^a,\varpi^a)$, and
hence $d(a,b,c) \in Z$. By Lemma \ref{normalizerlemma} we now have
$T(0,1,2)T(a,b,c) = T(a,a+1,2a+2) = T(a,b+1,c+2)$. This proves the theorem
in this case.

For the remainder of the proof we assume that $c>a+b$.
If $Q \in \mathcal{H}$, then we will write the expression of $Q$
in the $\Z$-basis $\{C\}_{C \in K \backslash \Delta / K}$ for $\mathcal{H}$
as
\begin{equation}
Q = \sum_{C \in K \backslash \Delta /K} j(Q,C) C.
\end{equation}
By Theorem \ref{degree1multiplicationthm}
\begin{align}
T(0,0,1)^2 & =
T(0,0,2) +(q+1)T(0,1,2) +(q+1)T(1,0,2)\nonumber \\
&\quad +(q^3+2q^2 +q) T(1,1,2) +(q-1)T(w) T(0,0,1).\label{degree2multiplicationthmeq3}
\end{align}
Multiplying \eqref{degree2multiplicationthmeq3} on the right by $T(a,b,c)$
we obtain
\begin{equation}
\label{degree2multiplicationthmeq4}
R = R_1 +(q+1) R_2 + (q+1) R_3 + (q^3+2q^2+q) R_4 + (q-1) R_5
\end{equation}
with
$R = T(0,0,1)^2 T(a,b,c)$,
$R_1 = T(0,0,2)T(a,b,c)$,
$R_2 = T(0,1,2)T(a,b,c)$,
$R_3 = T(1,0,2)T(a,b,c)$,
$R_4 = T(1,1,2) T(a,b,c)$,
and
$R_5 = T(w) T(0,0,1) T(a,b,c)$.
It follows that
\begin{align}
j(R,C) &= j(R_1,C) +(q+1) j(R_2,C) + (q+1) j(R_3,C)\nonumber \\
&\quad + (q^3+2q^2+q) j(R_4,C) + (q-1) j(R_5,C) \label{degree2multiplicationthmeq5}
\end{align}
for $C \in K \backslash \Delta / K$.
By \eqref{doublemulteq} the integers $j(R_i,C)$ are non-negative for $C \in
K \backslash \Delta / K$ and $1 \leq i \leq 5$.
This implies that for $C \in K \backslash \Delta /K$,
\begin{equation}
\label{degree2multiplicationthmeq6}
j(R_2, C) \neq 0 \implies
j(R,C) \neq 0.
\end{equation}
Now consider $R$. Applying  Theorem \ref{degree1multiplicationthm} twice
to calculate $R$, we find that $C \in K \backslash \Delta /K$,
\begin{equation}
\label{degree2multiplicationthmeq7}
j(R, C) \neq 0
\implies
\text{$C = Kw^\delta d(e,f,g)K$ for some $(\delta,(e,f,g)) \in A$}
\end{equation}
where
\begin{align}
A &=\left(\{0,1\} \times S_\Delta\right) \cap \left(\{(0,(a+i,b+j,c+ 2)): 0 \leq i,j \leq 2\} \right.\nonumber \\
&\quad \cup \left.\{(1,(a+i,b+j,c+1)): 0 \leq i,j \leq 1 \} \right).
\end{align}
By \eqref{degree2multiplicationthmeq6} and \eqref{degree2multiplicationthmeq7} we now have
\begin{equation}
\label{degree2multiplicationthmeq8}
R_2 = \sum_{(\delta,(e,f,g)) \in A}
j(R_2, Kw^\delta d(e,f,g) K)\cdot Kw^\delta d(e,f,g) K.
\end{equation}
Using \eqref{doublemulteq2}, the coset decomposition from Lemma~6.1.2 of \cite{RS},
and arguments by contradiction, one can verify that
\begin{equation}
\label{degree2multiplicationthmeq9}
m(T(0,1,2)T(a,b,c); Kw^\delta d(e,f,g) K)=0
\end{equation}
for
\begin{align*}
&(\delta,(e,f,g))  \in
\left(\{0,1\} \times S_\Delta\right) \cap \{ (0,(a,b,c+2)),\ (0,(a,b+2,c+2)), \\
&\qquad (0,(a+1,b,c+2)),\  (0,(a+1,b+2,c+2)),\ (0,(a+2,b,c+2)),\\
&\qquad (0,(a+2,b+2,c+2)),\ (1,(a,b+1,c+1)) \}.
\end{align*}
For this, it is useful to recall that by assumption $c>a+b$.
Since $j(R_2,C) = m(T(0,1,2)T(a,b,c);C)$ for $C \in K \backslash \Delta / K$, from
\eqref{degree2multiplicationthmeq8} and \eqref{degree2multiplicationthmeq9}
we obtain
$$
T(0,1,2)T(a,b,c) = R_2
=
\sum_{(\delta,(e,f,g)) \in A'}
j(R_2, Kw^\delta d(e,f,g) K)\cdot Kw^\delta d(e,f,g) K
$$
where
\begin{align*}
A' &=  \left(\{0,1\} \times S_\Delta \right) \cap
\{(0,(a,b+1,c+2)),\  (0,(a+1,b+1,c+2)), \\
&\quad (0,(a+2,b+1,c+2),\  (1,(a,b,c+1)),\ (1,(a+1,b,c+1)),\\
&\quad (1,(a+1,b+1,c+1)) \}.
\end{align*}
This proves the existence
of integers $m_i$ as in \eqref{degree2multiplicationthmeq1}
and \eqref{degree2multiplicationthmeq2}.
The uniqueness of the $m_i$ follows the fact that the
$Kw^\delta d(e,f,g) K$ for $(\delta,(a,b,c)) \in \{0,1\} \times S_\Delta$
form a  $\Z$-basis for $\mathcal{H}$ (recall the definition
of $\mathcal{H}$ and see Lemma \ref{BTdecomplemma}).
Finally, the $m_i$ are explicitly calculated using \eqref{doublemulteq2}
and the coset decomposition from Lemma~6.1.2 of \cite{RS}.
\end{proof}

\begin{theorem}
\label{gentheorem}
The paramodular Hecke algebra $\mathcal{H}$
is generated by the elements $X, Y_1, Y_2$, and $V$.
\end{theorem}

\begin{proof}
In this proof, if $L$ is a sequence of elements of $\mathcal{H}$,
then $\langle L \rangle \subset \mathcal{H}$ will denote all finite
$\Z$-linear combinations of the elements from $L$.
Let $\mathcal{H}'$ be the subring of $\mathcal{H}$ generated by
$X, Y_1, Y_2$, and $V$. To prove the theorem it will suffice
to prove that $\mathcal{H}_c \subset \mathcal{H}'$
for $c \in \Z_{\geq 0}$. In this proof we will use that
$\alpha \left( \mathcal{H}'\right) = \mathcal{H}'$; this follows from
\eqref{alphageneq}.

We first show that $\mathcal{H}_c\subset \mathcal{H}'$ for $0 \leq c\leq3$.
By Lemma \ref{BTdecomplemma} we have $\mathcal{H}_0 = \langle 1 \rangle$ and
$\mathcal{H}_1 = \langle V, X \rangle$; hence, $\mathcal{H}_0,
\mathcal{H}_1 \subset \mathcal{H}'$. By Lemma \ref{BTdecomplemma}, \eqref{Vouteq},
and \eqref{V2eq}
we have
$$
\mathcal{H}_2 =
\langle
T(0,0,2),\,Y_1,\,Y_2,\,V^2,\, VX
\rangle.
$$
By Theorem~\ref{degree1multiplicationthm} we have
$$
X^2 = T(0,0,2)+(q+1)(Y_1+Y_2) +(q^3+2q^2+q)V^2 + (q-1)VX.
$$
This equation implies that $T(0,0,2) \in \mathcal{H}'$,
so that $\mathcal{H}_2 \subset \mathcal{H}'$.
By Lemma \ref{BTdecomplemma}, \eqref{Vouteq}, \eqref{V2eq},
and \eqref{inductionfacteq} we have
$$
\mathcal{H}_3
=
\langle
T(0,0,3),\, T(0,1,3),\, T(1,0,3),\, V^2X,\, VT(0,0,2),\, VY_1,\, VY_2,\, V^3
\rangle.
$$
By Theorem~\ref{degree1multiplicationthm},
$$
XY_1=T(0,1,3)+q^2V^2X+(q^2-1)VY_1.
$$
It follows that $T(0,1,3) \in \mathcal{H}'$, and hence also $T(1,0,3)
= \alpha \left( T(0,1,3) \right)
\in \mathcal{H}'$. By Theorem~\ref{degree1multiplicationthm}
$$
XT(0,0,2)=T(0,0,3)+qT(0,1,3)+q T(1,0,3)+q^3 V^2 X+(q^2-1)VT(0,0,2).
$$
Since $T(0,0,2), T(0,1,3),T(1,0,3) \in \mathcal{H}'$, this implies that
$T(0,0,3) \in \mathcal{H}'$, completing the argument that $\mathcal{H}_3
\subset \mathcal{H}'$.

Now assume that $c \in \Z$ is such that $c \geq 4$ and $\mathcal{H}_k \subset
\mathcal{H}'$ for $0 \leq k <c$; we need
to prove that $\mathcal{H}_c \subset \mathcal{H'}$.
By Lemma~\ref{BTdecomplemma}, \eqref{Vouteq}, \eqref{inductionfacteq},
 \eqref{alphabetapropeq}, and $\alpha\left(\mathcal{H}'\right) =\mathcal{H}'$,
to prove $\mathcal{H}_c \subset \mathcal{H'}$
it suffices to show that $T(0,b,c)\in\mathcal{H}'$ for $b \in \Z$ with
$0\leq b\leq c-b$.
By Theorem~\ref{degree2multiplicationthm} and \eqref{inductionfacteq}, we have
\begin{align*}
Y_1T(0,0,c-2)&=T(0,1,c)+(q^2-q)V^2 T(0,0,c-2)+m_3 V^2 T(1,0,c-2)\\
&\quad +(q^2-q)VT(1,0,c-1),
\end{align*}
with $m_3=q^3+q^2$ if $c=4$ and $m_3=q^3$ if $c>5$.
Since $\mathcal{H}_k \subset \mathcal{H}'$ for $0 \leq k < c$ this equation
implies that
$T(0,1,c) \in \mathcal{H}'$. Next, by Theorem~\ref{degree1multiplicationthm} and
\eqref{inductionfacteq},
\begin{align*}
XT(0,0,c-1)&=T(0,0,c)+q T(0,1,c)+q \alpha( T(0,1,c)) \\
&\quad +q^3 V^2 T(0,0,c-2)+(q-1) V T(0,0,c-1).
\end{align*}
This implies that $T(0,0,c)\in \mathcal{H}'$. Finally, assume that
$2 \leq b \leq c-b$.  By Theorem~\ref{degree2multiplicationthm}
and \eqref{inductionfacteq}
we have
\begin{align*}
&Y_1T(0,b-1,c-2)=T(0,b,c)+m_2V^2 T(0,b-1,c-2)\\
&\qquad+m_3 V^2 T(1,b-1,c-2) +m_5VT(1,b-1,c-1)+m_6VT(1,b,c-1)
\end{align*}
where the integers $m_i$ are as in Theorem~\ref{degree2multiplicationthm}.
All the terms of this equation other than  $T(0,b,c)$ are in $\mathcal{H}'$
because $\mathcal{H}_k \subset \mathcal{H}'$ for $0 \leq k < c$; it follows
that $T(0,b,c) \in \mathcal{H}'$, completing the proof.
\end{proof}

\begin{lemma}
\label{relations}
We have 
\begin{align}
0&=V X - XV,\label{r1eq}\\
0&=VY_1-Y_2 V,\label{r2eq}\\
0&=V Y_2 - Y_1 V,\label{r3eq}\\
0&=XY_1-Y_1X +(1-q^2) VY_1-(1-q^2) VY_2,\label{r4eq}\\
0&=XY_2 -Y_2 X -(1-q^2) VY_1+(1-q^2) V Y_2,\label{r5eq}\\
0&=Y_1Y_2-Y_2Y_1 -(q-1)^2 (q+1) V^2Y_1 +(q-1)^2(q+1)V^2Y_2\nonumber \\
&\quad+(q-1)VXY_1-(q-1)VXY_2,\label{r6eq}\\
0&=Y_1Y_2+q^2(q+1)^2 V^4 + q(q^2-1) V^3 X - q V^2 X^2 -(q-1) VXY_2 \nonumber 
\\
&\quad+q(q+1)V^2 Y_1+(1+q^3) V^2 Y_2. \label{r7eq}
\end{align}
\end{lemma}
\begin{proof}
These equations follow from Theorem \ref{degree1multiplicationthm} and Theorem \ref{degree2multiplicationthm}.
\end{proof}

Let $R$ be the free $\Z$-algebra on the non-commuting variables $v,x,y_1,$ and 
$y_2$ (for this concept, see for example (1.2), p.~6 of \cite{Lam}).
We define a (weighted) degree function $d:R-0\to \Z_{\geq 0}$ in the
usual way by declaring $d(v)=d(x)=1$ and $d(y_1)=d(y_2)=2$. 
We let $I$ be the
two-sided ideal of $R$ generated by 
\begin{align}
r_1&=v x - xv,\\
r_2&=vy_1-y_2 v,\\
r_3&=v y_2 - y_1 v,\\
r_4&=xy_1-y_1x +(1-q^2) vy_1-(1-q^2) vy_2,\\
r_5&=xy_2 -y_2 x -(1-q^2) vy_1+(1-q^2) v y_2,\\
r_6&=y_1y_2-y_2y_1 -(q-1)^2 (q+1) v^2y_1 +(q-1)^2(q+1)v^2y_2\nonumber \\
&\quad+(q-1)vxy_1-(q-1)vxy_2,\\
r_7&=y_1y_2+q^2(q+1)^2 v^4 + q(q^2-1) v^3 x - q v^2 x^2 -(q-1) vxy_2 \nonumber 
\\
&\quad+q(q+1)v^2 y_1+(1+q^3) v^2 y_2.
\end{align}
For $m \in R$ we write $\bar m =m+I \in R/I$.
\begin{lemma}
\label{RIspanlemma}
The $\Z$-algebra $R/I$ is spanned over $\Z$ by the elements $\bar v^a \bar x^b 
\bar y_1^{c_1} \bar y_2^{c_2}$ such that $a,b,c_1,c_2 \in \Z_{\geq 0}$ and 
$c_1=0$ or $c_2=0$. 
\end{lemma}
\begin{proof}
Let $M$ be the $\Z$-submodule of $R/I$ spanned by the elements in the 
statement. For $k \in \Z_{\geq 0}$ let $(R/I)(k)$ be the $\Z$ submodule of $R/I$ 
spanned by the elements $\bar r$ where $r$ is a monomial in $v, x, y_1, y_2$ 
and $d(r)=k$. Then $R/I$ is the sum of the $(R/I)(k)$ for $k \in \Z_{\geq 0}$. 
To prove the lemma it will suffice to prove that $(R/I)(k) \subset M$ for all 
$k \in \Z_{\geq 0}$. Clearly $(R/I)(0) \subset M$. Assume that $k \geq 0$ and 
$(R/I)(k) \subset M$; we will prove that $(R/I)(k+1) \subset M$. We first note 
that $(R/I)(k) \subset M$ and $\bar r_1 =\bar r_2 =\bar r_3=0$ imply the 
following: (S) If $m \in R$ is a monomial in $v,x,y_1,y_2$ with 
$d(m)=k+1$ and $v$ occurs in $m$, then $\bar m \in M$. Now let $q \in R$ be a 
monomial in $v,x,y_1,y_2$ with $d(q)=k+1$. Using $\bar r_1=\bar r_2= \bar 
r_3=\bar r_4 = \bar r_5 = \bar r_6 =0$ and (S) we see that $\bar q +M= \bar v^a 
\bar x^b \bar y_1^{c_1} \bar y_2^{c_2}+M$ for some $a,b,c_1,c_2 \in \Z_{\geq 
0}$.  Finally, using $\bar r_7=0$ and (S), we have $\bar v^a 
\bar x^b \bar y_1^{c_1} \bar y_2^{c_2} \in M$ so that $\bar q \in M$. 
\end{proof}

\begin{theorem}
\label{genreltheorem}
If $k \in \Z_{\geq 0}$, then 
$$
\mathcal{B}_k=\left\{ V^{e_1} X^{e_2} Y_1^{e_3} Y_2^{e_4} :  e_1,e_2,e_3,e_4 
\in \Z_{\geq 0},\
e_1+e_2+2(e_3+e_4)=k,\ e_3e_4=0\right\}
$$
is a $\Z$ basis for $\mathcal{H}_k$. 
The homomorphism $i: R \to \mathcal{H}$ determined by $i(v)=V$, $i(x)=X, 
i(y_1)=Y_1$, and $i(y_2)=Y_2$  is surjective and has kernel $I$ so that 
$ R/I \cong \mathcal{H}$.
\end{theorem}
\begin{proof}
The homomorphism $i$ is surjective by Theorem \ref{gentheorem}, and $I \subset \ker(i)$ by
Lemma \ref{relations}. Let $\bar i: R/I \to \mathcal{H}$ be the homomorphism induced by
$i$; we need to prove that $i$ is injective. Let $k \in \Z_{\geq 0}$. We 
first claim that
$\mathcal{B}_k$
is a $\Z$-basis for $\mathcal{H}_k$. A calculation 
shows that the number of elements in $\mathcal{B}_k$ is 
$\mathrm{rank}_\Z(\mathcal{H}_k)$ (see \eqref{rankeq}). Thus, it suffices to prove that $\mathcal{B}_k$
spans $\mathcal{H}_k$ over $\Z$. Let $\mathcal{H}_k'$ be the $\Z$ submodule of 
$\mathcal{H}_k$ spanned by $\mathcal{B}_k$. Since $\bar i$ is surjective, Lemma 
\ref{RIspanlemma} implies that $\mathcal{H} = \sum_{k=0}^\infty \mathcal{H}_k'$. Also,  since
$\mathcal{H}_k' \subset \mathcal{H}_k$ for $k \in \Z_{\geq 0}$ and $\mathcal{H} 
= \oplus_{k=0}^\infty \mathcal{H}_k$, the sum $\sum_{k=0}^\infty 
\mathcal{H}_k'$ is direct so that $\mathcal{H} = \oplus_{k=0}^\infty 
\mathcal{H}_k'$. We now deduce that $\mathcal{H}_k' = \mathcal{H}_k$ for $k \in 
\Z_{\geq 0}$, proving that $\mathcal{B}_k$ is a $\Z$-basis for $\mathcal{H}_k$ 
for $k \in \Z_{\geq 0}$. Next, let 
$$
\mathcal{C}_k=\left\{ \bar v^a \bar x^b \bar y_1^{c_1} \bar y_2^{c_2} :  a,b,c 
\in \Z_{\geq 0},\
a+b+2c_1+2c_2=k,\  c_1=0\ \text{or}\ c_2=0\right\},
$$
and let $M_k$ be the $\Z$-submodule spanned by $\mathcal{C}_k$ in $R/I$  for 
$k \in \Z_{\geq 0}$.  By Lemma \ref{RIspanlemma} we have $R/I = \sum_{k=0}^\infty M_k$.
Let $m \in \ker(\bar i)$. Write $m = 
\sum_{k=0}^\infty m_k$ where $m_k \in M_k$ for $k \in \Z_{\geq 0}$ and $m_k=0$ 
for all but finitely many $k$. Then $0 = \sum_{k=0}^\infty \bar i (m_k)$. As 
$\mathcal{H}$ is the direct sum of the $\mathcal{H}_k$ and $\bar i(M_k) 
=\mathcal{H}_k$ for $k \in \Z_{\geq 0}$, we have $\bar i (m_k) =0$ for $k \in 
\Z_{\geq 0}$. Let $k \in \Z_{\geq 0}$. Let $p_1,\dots,p_N$ be a listing of the 
elements in $\mathcal{C}_k$, and write $m_k = \sum_{\ell=1}^N n_\ell p_\ell$ 
for some $n_1,\dots,n_N \in \Z$. Then $0 = \sum_{\ell=0}^N n_\ell \bar i 
(p_\ell)$. Now  $\bar i(p_1),\dots, \bar i(p_N)$ is a listing 
of the elements of $\mathcal{B}_k$. Since $i(p_1),\dots,i(p_N)$ is a 
$\Z$-basis for $\mathcal{H}_k$ we must have $n_1=\cdots = n_N=0$. Hence, 
$m_k=0$, so that $m=0$. 
\end{proof}

The first few $\mathcal{B}_k$ are
$$
\begin{array}{ccl}
\toprule
k & \mathrm{rank}_\Z(\mathcal{H}_k) & \mathcal{B}_k \\
\midrule
0 & 1 & 1 \\
1 & 2 & V,\: X\\
2 & 5 & Y_1,\: Y_2,\: V^2,\: VX,\: X^2\\
3 & 8 & VY_1\:, XY_1,\: VY_2,\: XY_2,\: V^3,\: V^2X,\: VX^2,\: X^3\\
4 & 13 & Y_1^2,\: Y_2^2,\: V^2 Y_1,\: VXY_1,\: X^2 Y_1,\: V^2 Y_2,\: VX Y_2, \:
X^2 Y_2,\\
&& V^4,\: V^3 X,\: V^2 X^2,\: V X^3,\: X^4\\
\bottomrule
\end{array}
$$

Since $\mathcal{B}_k$ is a basis for $\mathcal{H}_k$
for $k \in \Z_{\geq 0}$ by Theorem \ref{genreltheorem},
Lemma \ref{relations} implies that $\mathcal{H}$ is not commutative.
In particular, the commutators $[V,Y_1]=-[V,Y_2]$,
$[X,Y_1]=-[X,Y_2]$, and $[Y_1,Y_2]$ are all non-zero.

It is interesting to observe that  the Hilbert-Poincar\'e series of $\mathcal{H}$ satisfies the statement of the Hilbert-Serre theorem:
if $t$ is an indeterminate, then
\begin{equation}
\label{hilbertserieseq}
\sum_{k=0}^\infty \mathrm{rank}_\Z (\mathcal{H}_k ) t^k =
\dfrac{1-t^4}{(1-t)(1-t)(1-t^2)(1-t^2)}.
\end{equation}
For this use \eqref{rankeq}.
For a statement of the usual Hilbert-Serre theorem see Theorem~11.1 of \cite{AM}
or \cite{Stan}.
The Hilbert-Serre theorem does not generally hold for graded non-commutative algebras. See for example \cite{She} and \cite{A}.
Also, we note that $\mathcal{H}$ is not a domain, i.e., there exist $P, Q \in \mathcal{H}$ with $P \neq 0$, $Q \neq 0$, and
$PQ=0$. For example,
\begin{equation}
\label{zerodiveq}
\left( Y_1  + (1+q) V^2 +VX\right) \left( Y_2  +(q^2+q^3) V^2-q VX \right) =0.
\end{equation}
In the next section we will
use \eqref{zerodiveq} to classify homomorphisms from $\mathcal{H}$ to $\C$.

To conclude this section we relate $\mathcal{H}$ to the Hecke algebra studied in \cite{GK}.
For the remainder of this section we assume that $p$ is a prime of $\Z$ and that $F=\Q_p$.
Let
$$
\mathrm{K}(p) =
\SSp(4,\Q) \cap
\begin{bsmallmatrix}
\Z & \Z & p^{-1} \Z & \Z \\
p \Z & \Z & \Z & \Z \\
p \Z & p \Z & \Z & p \Z \\
p \Z & \Z & \Z & \Z
\end{bsmallmatrix}.
$$
We refer to $\mathrm{K}(p)$ as the \emph{paramodular congruence group} of level $p$.
In \cite{GK} the group $\mathrm{K}(p)$ is denoted by $\Sigma_p$ (\cite{GK} actually
considers a different realization of $\mathrm{K}(p)$, but this difference is
insignificant). We also define
\begin{align*}
\Delta(p) & = \{ g \in \GSp(4,\Q) \cap \begin{bsmallmatrix}
\Z & \Z & p^{-1} \Z & \Z \\
p \Z & \Z & \Z & \Z \\
p \Z & p \Z & \Z & p \Z \\
p \Z & \Z & \Z & \Z
\end{bsmallmatrix} : \text{$\lambda(g) = p^k$ for some $k \in \Z_{\geq 0}$} \},\\
\mathbb{G}_p & = \{ \lambda(h)^{-\frac{1}{2}} h: \text{$h \in \GSp(4,\R) \cap \Mat(4,\Z)$
and $\lambda(g) = p^k$ for some $k \in \Z_{\geq 0}$} \}.
\end{align*}
Then $\Delta(p)$ and $\mathbb{G}_p$ contain $1$ and are semi-groups.
One may verify that the commensurators of $\mathrm{K}(p)$ in  $\GSp(4,\Q)^+
=\{g \in \GSp(4,\Q): \lambda(g)>0\}$ and $\R_{>0}\cdot \GSp(4,\Q)^+$ are
$\GSp(4,\Q)^+$ and $\R_{>0}\cdot \GSp(4,\Q)^+$, respectively. Since $\mathrm{K}(p) \subset
\Delta(p) \subset \GSp(4,\Q)^+$ and $\mathrm{K}(p) \subset \mathbb{G}_p \subset
\R_{>0} \cdot \GSp(4,\Q)^+$, we may consider the Hecke algebras
 $\mathcal{R}(\mathrm{K}(p), \Delta(p))$ and $\mathcal{R}(\mathrm{K}(p),\mathbb{G}_p)$
(see p.~54 of \cite{S2}). One may prove that there is an isomorphism
of rings
$$
\mathcal{R}(\mathrm{K}(p), \Delta(p)) \stackrel{\sim}{\longrightarrow}
\mathcal{R}(K,\Delta) = \mathcal{H}
$$
that sends $\mathrm{K}(p)g\mathrm{K}(p)$ to $KgK$ for $g \in \Delta(p)$.
Thus $\mathcal{R}(\mathrm{K}(p), \Delta(p))$
is essentially the Hecke algebra $\mathcal{H}$ studied in this paper when $F=\Q_p$.
The Hecke algebra $\mathcal{R}(\mathrm{K}(p),\mathbb{G}_p)$ is considered in
\cite{GK}, and is denoted there by $\mathcal{H}(\Sigma_p,\mathbb{G}_p)$. The
following proposition relates $\mathcal{R}(\mathrm{K}(p),\mathbb{G}_p)$ to
$\mathcal{R}(\mathrm{K}(p), \Delta(p))$.
\begin{proposition}
\label{GKprop}
Define $f: \mathcal{R}(\mathrm{K}(p), \Delta(p)) \to \mathcal{R}(\mathrm{K}(p),\mathbb{G}_p)$
by setting
$$
f(\mathrm{K}(p)g\mathrm{K}(p)) = \mathrm{K}(p)\lambda(g)^{-\frac{1}{2}}g\mathrm{K}(p) \quad
\text{for $g \in \Delta(p)$}
$$
and linearly extending to $\mathcal{R}(\mathrm{K}(p), \Delta(p))$.
Then $f$ is a well-defined surjective ring homomorphism with kernel the two-sided ideal
generated by $V^2-1$.
\end{proposition}
We leave the proof of this proposition to the reader. Theorem 5.5 of \cite{GK}
asserts that the ring $\mathcal{R}(\mathrm{K}(p),\mathbb{G}_p)$ is generated by
$f(V)$, $f(X)$, and $f(Y_1)$. This result can also be deduced from
Theorem \ref{gentheorem}, \eqref{r2eq}, and Proposition \ref{GKprop} (note
that $f(V)^2=1$).

\section{Complex homomorphisms and rationality}

Let $\Hom(\mathcal{H}, \C)$
be the set of all homomorphisms from $\mathcal{H}$ to $\C$, i.e., 
maps $\chi: \mathcal{H} \to \C$  such that $\chi(1)=1$, $\chi(xy) = \chi(x) 
\chi(y)$, and 
$\chi(x+y) = \chi(x)+\chi(y)$ for $x, y \in \mathcal{H}$. 
If $(a,b,c) \in D_\Delta$, then we abbreviate
\begin{equation}
\chi (a,b,c) = \chi(T(a,b,c)).
\end{equation}
Also, for $k \in \Z_{\geq 0}$, let
\begin{equation}
T(q^k) = \sum\limits_{\substack{C \in K \backslash \Delta / K,\\ C \in \mathcal{H}_k}}
C.
\end{equation}
In this section, for a natural set of $\chi \in \Hom(\mathcal{H},\C)$, we
will calculate the formal power series $\sum_{k=0}^\infty \chi(T(q^k)) t^k$
for an indeterminate $t$. We begin by classifying the elements of $\Hom(\mathcal{H},\C)$.

\begin{lemma}
\label{charactertypelemma}
If $\chi \in \Hom(\mathcal{H}, \C)$, then one of the following holds:
\begin{enumerate}
\item[(i)] $\chi(V) \neq 0$ and $\chi(Y_1) = \chi(Y_2) = -(1+q) \chi(V)^2 
-\chi(V) \chi(X)$. 
\item[(ii)] $\chi(V) \neq 0$ and $\chi(Y_1) = \chi(Y_2) = -(q^2+q^3) \chi(V)^2 
+ q\chi(V) \chi(X)$. 
\item[(iii)] $\chi(V) =0$ and $\chi(Y_1) \chi(Y_2)=0$. 
\end{enumerate}
Let $\Hom(\mathcal{H},\C)_1, \Hom(\mathcal{H},\C)_2,$ and 
$\Hom(\mathcal{H},\C)_3$ be the sets of $\chi \in \Hom(\mathcal{H},\C)$ 
satisfying (i), (ii), and (iii) respectively, and define
\begin{align*}
V_1 & = \{(\varepsilon,\lambda,\mu) \in \C^\times \times \C \times \C: \mu =-(1+q)
\varepsilon^2 -\varepsilon \lambda \},\\
V_2 & = \{(\varepsilon,\lambda,\mu) \in \C^\times \times \C \times \C: \mu =
-(q^2+q^3) \varepsilon^2+ q \varepsilon \lambda  \},\\
V_3 & = \{ (\lambda, \mu_1,\mu_2) \in \C^3: \mu_1 
\mu_2 =0\}.
\end{align*}
Then the functions
\begin{align}
&\Hom(\mathcal{H},\C)_1
\stackrel{\sim}{\longrightarrow}
V_1, \qquad 
\chi \mapsto (\chi(V),\chi(X), \chi(Y_1)),\label{cteq1} \\
&\Hom(\mathcal{H},\C)_2
\stackrel{\sim}{\longrightarrow}
V_2, \qquad 
\chi \mapsto (\chi(V),\chi(X), \chi(Y_1)),\label{cteq2} \\
&\Hom(\mathcal{H},\C)_3
\stackrel{\sim}{\longrightarrow}
V_3, \qquad 
\chi \mapsto (\chi(X),\chi(Y_1), \chi(Y_2)) \label{cteq3}
\end{align}
are bijections. 
\end{lemma}
\begin{proof} 
Let $\chi \in \Hom(\mathcal{H},\C)$. 
Applying $\chi$ to \eqref{r2eq} gives $\chi(V)(\chi(Y_1)-\chi(Y_2))=0$. Hence, 
$\chi(V)=0$ or $\chi(V) \neq 0$ and $\chi(Y_1) =\chi(Y_2)$. Assume $\chi(V)=0$. 
Applying $\chi$ to \eqref{r7eq} implies that $\chi(Y_1)\chi(Y_2)=0$ so that 
(iii) holds. Assume that $\chi(V) \neq 0$ and $\chi(Y_1) = \chi(Y_2)$. Applying 
to $\chi$ to \eqref{zerodiveq}  we obtain
$$
0 =
\left(\chi(Y_1) +(1+q) \chi(V)^2+ \chi(V) \chi(X) \right)
 \left( \chi(Y_1) + (q^2+q^3) \chi (V)^2 - q \chi(V) \chi(X)
\right).
$$
Hence, (i) or (ii) holds. The functions \eqref{cteq1}, 
\eqref{cteq2}, and \eqref{cteq3} are injective by Theorem~\ref{gentheorem}. Let
$i =1$ or $i=2$. To prove that $\Hom(\mathcal{H},\C)_i \to V_i$ is surjective, 
let $(\varepsilon,\lambda, \mu) \in V_i$. Let $\eta: R \to \C$ be the
homomorphism determined by setting
$\eta(v) = \varepsilon$, $\eta(x) = \lambda$, and $\eta(y_1) = \eta(y_2) = \mu$
($R$ is defined after Lemma \ref{relations}).
Calculations show that $\eta(r_j) =0$ for $j=1,\dots,7$, so that $\eta(I)=0$. 
By Theorem~\ref{genreltheorem}, $\eta$ induces a homomorphism $\chi:
\mathcal{H} \to \C$ such that $\chi \in \Hom(\mathcal{H},\C)_1$ and
$\chi(V)=\varepsilon$, $\chi(X) = \lambda$, and $\chi(Y_1)=\mu$. The proof that 
$\Hom(\mathcal{H},\C)_3 \to V_3$ is surjective is similar. 
\end{proof}

By Lemma \ref{charactertypelemma} we see that $\Hom(\mathcal{H},\C)_3$ has
empty intersection with $\Hom(\mathcal{H},\C)_1$ and $\Hom(\mathcal{H},\C)_2$,
and
$$
\Hom(\mathcal{H},\C)_1 \cap \Hom(\mathcal{H},\C)_2
=
\{ \chi \in \Hom(\mathcal{H},\C): \text{$\chi(V) \neq 0$ and
$\chi(X) = \chi(V) (q^2-1)$} \}.
$$

Representation theory provides examples of elements of $\Hom(\mathcal{H},\C)_1$
and $\Hom(\mathcal{H},\C)_2$. Assume that
$(\pi,V)$ is an irreducible, admissible representation of $G$ with trivial central
character. As in \cite{RS}, define
\begin{align}
V(0) & = \{ v \in V: \pi(k) v = v, k \in \GSp(4,\OF)\},\label{level0eq}\\
V(1) & = \{ v \in V: \pi(k) v = v, k \in K\}.\label{level1eq}
\end{align}
The algebra $\mathcal{H}$ acts on $V(1)$, and we also denote this action by $\pi$ (see p.~188 of
\cite{RS} and use \eqref{Hprimeisoeq}). Assume
for the rest of this paragraph that $\dim V(1) =1$, and let $V(1)$ be spanned by $v$.
Then there exists $\chi=\chi_\pi \in \Hom(\mathcal{H},\C)$ such that $\pi(x)v = \chi(x)v$ for
$x \in \mathcal{H}$.
Since $\pi$ has trivial central character we have $\chi(1,1,2)=1$.
Using Tables A.12 and A.13 of \cite{RS}, one may verify
that
\begin{align}
V(0)=0 & \iff \chi \in \Hom(\mathcal{H},\C)_1, \label{pitype1eq}\\
V(0) \neq 0 & \iff \chi \in \Hom(\mathcal{H},\C)_2. \label{pitype2eq}
\end{align}
For this verification we note that if $V(0)=0$, so that $v$ is a newform in the
terminology of \cite{RS}, then $\chi(X) = \lambda_\pi$, $\chi(Y_1) = \chi(Y_2)=\mu_\pi$, and
$\chi(V) = \varepsilon_\pi$; if $V(0) \neq 0$, then $\chi(X) = \lambda_\pi+q^2-1$,
$\chi(Y_1) = \chi(Y_2)=q\lambda_\pi-q(q+1)$, and $\chi(V) =1$ (for this second case
it is useful to use some formulas from the proof of Lemma 2.3.6 of \cite{JLRS}).
Here, $\lambda_\pi, \mu_\pi$, and
$\varepsilon_\pi$ are as in \cite{RS}. In the notation of \cite{RS}, the $\pi$ as in \eqref{pitype1eq} are the
IIa, IVc, Vb, Vc, and VIc representations with unramified inducing data, while
the $\pi$ as in \eqref{pitype2eq} are the IIb, IVd, and VId representations with
unramified inducing data.

In addition, a well-known example of an element of $\Hom(\mathcal{H},\C)_2$ is the \emph{index homomorphism}
$\chi_{\mathrm{index}}$. The homomorphism $\chi_{\mathrm{index}}$ is defined as follows.
If $C \in K \backslash \Delta / K$ and $C = \sqcup_{\ell =1}^n K g_l$ is a disjoint
decomposition, then $\chi_{\mathrm{index}}(C) = n$; if $x = \sum_{C \in K \backslash \Delta /K}
n(C) C$ is an element of $\mathcal{H}$, then $\chi_{\mathrm{index}}$ is defined to be
$\chi_{\mathrm{index}}(x) = \sum_{C \in K \backslash \Delta / K} n(C) \chi_{\mathrm{index}}(C)$.
The function $\chi_{\mathrm{index}}: \mathcal{H} \to \C$ is a homomorphism by
Proposition~3.3 of \cite{S2}. We have
\begin{align}
\chi_{\mathrm{index}}(1,1,2) &=  \chi_{\mathrm{index}}(V^2) =1, \label{indexeq1}\\
\chi_{\mathrm{index}}(X) &=q^3+2 q^2 +q,\label{indexeq2} \\
\chi_{\mathrm{index}}(Y_1) &= \chi_{\mathrm{index}}(Y_2) = q^4+q^3 \label{indexeq3}
\end{align}
by Lemma~6.1.2 of \cite{RS}. In fact, $\chi_{\mathrm{index}} = \chi_\pi$ with $\pi$ the trivial
representation of $G$ (which is a IVd representation) and $\chi_\pi$ as in the previous paragraph.

Elements of $\Hom(\mathcal{H},\C)_3$ do not arise from representations of $\GSp(4,F)$; however,
we will use elements of $\Hom(\mathcal{H},\C)_3$ to help determine the center of $\mathcal{H}$ in the final
section of this paper.

We will now prove a series of lemmas required for the main result of this section.
For the next lemma we recall the automorphism $\alpha$ and the anti-automorphism
$\beta$ of $\mathcal{H}$ defined in section~\ref{structuresec}.

\begin{lemma}
\label{alphabetainvlemma} 
Let $\chi \in \Hom(\mathcal{H}, \C)$. 
\begin{enumerate}
\item[(i)] If $\chi \in \Hom(\mathcal{H},\C)_1$ or $\chi \in 
\Hom(\mathcal{H},\C)_2$, then $\chi \circ \alpha = \chi \circ \beta 
= \chi$, and hence $\chi(a,b,c) = \chi(b,a,c)$ for $(a,b,c) \in D_\Delta$.
\item[(ii)] If $\chi(1,1,2)=1$, then $\chi(V) 
\in \{\pm 1\}$ and $\chi(a,b,c)=\chi(a-1,b-1,c-2)$ for $(a,b,c) \in D_\Delta$
with $a, b \geq 1$, and $c \geq 2$. 
\end{enumerate}
\end{lemma}
\begin{proof} (i). Assume that $\chi \in \Hom(\mathcal{H},\C)_1$ or $\chi \in 
\Hom(\mathcal{H},\C)_2$. 
Then $\chi(Y_1) = \chi(Y_2)$. 
By Theorem~\ref{gentheorem} the algebra $\mathcal{H}$ is generated by
$V,X,Y_1,$ and $Y_2$, and by \eqref{alphageneq} and  \eqref{betageneq}, $(\chi \circ \alpha)( V )=\chi(V)=(\chi \circ \beta)(V)$,
$(\chi \circ \alpha) (X) =\chi(X)=(\chi \circ \beta)(X)$,
$(\chi \circ \alpha) (Y_1)=\chi(Y_2) =\chi(Y_1) = (\chi \circ \beta)(Y_1)$,
and
$(\chi \circ \alpha) (Y_2)=\chi(Y_1) =\chi(Y_2) = (\chi \circ \beta)(Y_2)$.
Hence, $\chi = \chi \circ
\alpha = \chi \circ \beta$.
The final statement follows from \eqref{alphabetapropeq}.

(ii). Assume that $\chi(1,1,2)=1$. By \eqref{V2eq},  $V^2 = T(1,1,2)$;
hence $\chi(V)^2 = \chi(1,1,2)=1$.  Let $(a,b,c) \in D_\Delta$ with $a,b
\geq 1$ and $c \geq 2$. Then $(a-1,b-1,c-2) \in
D_\Delta$. By \eqref{inductionfacteq}, $\chi(a,b,c) = \chi(T(1,1,2)T(a-1,b-1,c-2))
=\chi(1,1,2)\chi(a-1,b-1,c-2) = \chi(a-1,b-1,c-2)$. 
\end{proof}

\begin{lemma}
\label{aklemma}
Let $\chi \in \Hom(\mathcal{H},\C)_1$ or $\chi \in 
\Hom(\mathcal{H},\C)_2$ with $\chi(1,1,2)=1$ and define
$\varepsilon 
= \chi(V)$ (so that $\varepsilon=\pm 1$ by Lemma \ref{alphabetainvlemma}).
Define 
\begin{equation}
\label{akeq2}
a_k =a_k(\chi) = \chi (0,0,k) \qquad \text{for $k \in \Z_{\geq 0}$}. 
\end{equation}
Then 
\begin{equation}
\label{akeq3}
a_{k+4} = A a_{k+3} + B a_{k+2} + C a_{k+1} + D a_{k} \qquad \text{for $k 
\geq 1$,}
\end{equation}
where 
\begin{equation}
\label{akeq35}
A  = \varepsilon (1-q^2)  + \chi (X), \qquad
B =  -(q+q^3) + \varepsilon q(q-1) \chi (X)
-2q \chi(Y_1), \qquad
C  =q^3 A, \qquad
D = -q^6.
\end{equation}
\end{lemma}

\begin{proof}
In this proof we will repeatedly use Lemma~\ref{alphabetainvlemma}.
If $k \geq 2$, then by  Theorem~\ref{degree1multiplicationthm}
\begin{align*}
\chi(X) \chi(0,0,k)
&=
\chi(0,0,k+1) +2 q \chi(0,1,k+1)+q^3 \chi (0,0,k-1) \\
&\quad + (q-1) \chi (V) \chi (0,0,k),
\end{align*}
or equivalently,
\begin{align}
\label{akeq1}
\chi(0,1,k+1) & = a a_{k+1} +b  a_k +c a_{k-1}
\end{align}
where 
$a = -(2q)^{-1}$, $b =(2q)^{-1} ( \chi (X) - (q-1) \varepsilon)$, and $c 
= -2^{-1} q^2$.
Let $k \geq 3$. By  Theorem~\ref{degree2multiplicationthm}, and \eqref{akeq1},
\begin{align*}
\chi(Y_1) \chi (0,0,k)
& = \chi(0,1,k+2) + (q^2-q) \varepsilon \chi (0,1,k+1) \\
&\quad + q^3 \chi (0,1,k) +(q^2-q) \chi (0,0,k)\\
\chi(Y_1) a_k 
& = a a_{k+2} +b a_{k+1} +c a_{k}
+ (q^2-q) \varepsilon \left( aa_{k+1} + ba_k + ca_{k-1} \right) \\
&\quad + q^3 \left(aa_k +b a_{k-1} + ca_{k-2} \right) +(q^2-q) a_k\\
& = a a_{k+2} +\left( b +a(q^2-q) \varepsilon \right) a_{k+1} \\
&\quad +\left( c +(q^2-q) \varepsilon  b +q^3 a +(q^2-q) \right) a_{k}\\
&\quad +\left(  (q^2-q) \varepsilon c + q^3 b \right)a_{k-1}
+q^3c a_{k-2}.
\end{align*}
Hence,
\begin{align*}
a_{k+2} & =  -\left( b +a(q^2-q) \varepsilon \right) 
a^{-1}a_{k+1} \\
&\quad -\left( c +(q^2-q) \varepsilon  b +q^3 a +(q^2-q) -\chi(Y_1)  
\right)a^{-1} 
a_{k}\\
&\quad -\left(  (q^2-q) \varepsilon c + q^3 b \right)a^{-1}a_{k-1}
-ca^{-1}q^3 a_{k-2}\\
& = A a_{k+1} + B a_k + C a_{k-1} + Da_{k-2}.
\end{align*}
This proves \eqref{akeq3}.
\end{proof}

\begin{lemma}
\label{Bklemma}
Let $\chi \in \Hom(\mathcal{H},\C)_1$ or $\chi \in 
\Hom(\mathcal{H},\C)_2$ with  $\chi(1,1,2)=1$ and define  $\varepsilon
= \chi(V)$ (so that $\varepsilon=\pm 1$ by Lemma \ref{alphabetainvlemma}).
Define
\begin{equation}
\label{Bkeq1}
B_k =B_k(\chi)= \sum_{\substack{(a,b,k) \in S_\Delta, \\ ab=0}} \chi(a,b,k) \qquad
\text{for $k \in \Z_{\geq 0}$}.
\end{equation}
If $k \geq 2$, then 
\begin{align}
&B_{k+2} - (q+1)^{-1}\left( \varepsilon (1-q^2)+\chi(X) \right) 
B_{k+1} +(q+1)^{-1}(q^2+q^3)B_{k}\nonumber \\
&\qquad= (q+1)^{-1}  q  a_{k+2}
+(q+1)^{-1}     \varepsilon  (q^2 -q) a_{k+1}
-(q+1)^{-1}q^2a_{k} \label{Bkeq2}
\end{align}
where $a_k$ is defined as in \eqref{akeq2}.
\end{lemma}
\begin{proof}
Again, we will repeatedly use Lemma \ref{alphabetainvlemma}.
For $k \in \Z_{\geq 0}$  define
$$
C_k = \sum_{(a,0,k) \in S_\Delta} \chi (a,0, k) = \sum_{a=0}^{\lfloor k/2 \rfloor}
\chi 
(a,0, k).
$$
Then 
$$
B_k = 2 C_k -\chi(0,0,k)= 2 C_k -a_k \qquad \text{for $k \in \Z_{\geq 
0}$}.
$$
We first consider the $C_k$. Let $k \in \Z_{\geq 0}$ with $k \geq 3$. 
Let $\delta_k=0$ if $k$ is even and $\delta_k=1$ if $k$ is odd. 
Let $a \in \Z_{\geq 0}$ be such that $0 \leq a \leq \lfloor k/2 \rfloor$. 
Then by Theorem~\ref{degree1multiplicationthm}
\begin{align*}
\chi (X) \chi (a,0,k) 
& = n_1(a,0,k) \chi (a,0,k+1) + n_2(a,0,k) \chi(a,1,k+1) \\
&\quad +n_3(a,0,k) \chi (a+1,0,k+1) +n_4(a,0,k) \chi (a+1,1,k+1) \\
&\quad +n_5(a,0,k) \varepsilon \chi (a,0,k).
\end{align*}
It follows that
\begin{align*}
&\chi(X) C_k\\
&\qquad =  \sum_{a=0}^{\lfloor k/2 \rfloor}n_1(a,0,k) \chi (a,0,k+1) 
+\sum_{a=0}^{\lfloor k/2 \rfloor} n_2(a,0,k) \chi(a,1,k+1) \\
&\qquad\quad +\sum_{a=0}^{\lfloor k/2 \rfloor}n_3(a,0,k) \chi (a+1,0,k+1) 
+\sum_{a=0}^{\lfloor k/2 \rfloor}n_4(a,0,k) \chi (a+1,1,k+1) \\
&\qquad\quad +\varepsilon \sum_{a=0}^{\lfloor k/2 \rfloor}n_5(a,0,k)  \chi 
(a,0,k).
\end{align*}
By Theorem~\ref{degree1multiplicationthm},
\begin{align*}
&\sum_{a=0}^{\lfloor k/2 \rfloor}n_1(a,0,k) \chi (a,0,k+1) \\
&\qquad =-\delta_k \chi (\lfloor (k+1)/2\rfloor,0,k+1) 
+\sum_{a=0}^{\lfloor (k+1)/2 
 \rfloor} \chi (a,0,k+1)\\ 
 &\qquad =-\delta_k \chi (\lfloor (k+1)/2\rfloor,0,k+1) 
+C_{k+1}.
\end{align*}
And:
\begin{align*}
&\sum_{a=0}^{\lfloor k/2 \rfloor} n_2(a,0,k) \chi(a,1,k+1)\\
&\qquad =n_2(0,0,k) \chi(0,1,k+1) + \sum_{a=1}^{\lfloor k/2 \rfloor} 
n_2(a,0,k) \chi(a,1,k+1)\\
&\qquad =q \chi(0,1,k+1) + q^2\sum_{a=1}^{\lfloor k/2 \rfloor}\chi(a-1,0,k-1)\\
&\qquad =q \chi(0,1,k+1) + q^2\sum_{a=0}^{\lfloor k/2 
\rfloor-1}\chi(a,0,k-1)\\
&\qquad =q \chi(0,1,k+1) -\delta_k q^2 \chi (\lfloor (k-1)/2\rfloor,0,k-1) 
+ q^2\sum_{a=0}^{\lfloor (k-1)/2 
\rfloor}\chi(a,0,k-1)\\
&\qquad =q \chi(0,1,k+1) -\delta_k q^2 \chi (\lfloor (k-1)/2\rfloor,0,k-1) 
+ q^2C_{k-1}.
\end{align*}
And:
\begin{align*}
&\sum_{a=0}^{\lfloor k/2 \rfloor}n_3(a,0,k) \chi (a+1,0,k+1) \\
&\qquad = n_3(\lfloor k/2 \rfloor,0,k) \chi(\lfloor k/2 \rfloor+1,0,k+1)\\
&\qquad\quad + \sum_{a=0}^{\lfloor k/2 \rfloor-1}n_3(a,0,k) \chi (a+1,0,k+1) \\
&\qquad = \delta_k (q+1) \chi(\lfloor k/2 \rfloor+1,0,k+1)+ 
q\sum_{a=0}^{\lfloor 
k/2 \rfloor-1} \chi (a+1,0,k+1) \\
&\qquad = \delta_k (q+1) \chi(\lfloor k/2 \rfloor+1,0,k+1)+ 
q\sum_{a=1}^{\lfloor 
k/2 \rfloor} \chi (a,0,k+1) \\
&\qquad = \delta_k (q+1) \chi(\lfloor k/2 \rfloor+1,0,k+1)\\ 
&\qquad\quad-q \chi(0,0,k+1) -q \delta_k \chi(\lfloor (k+1)/2\rfloor,0,k+1)+
q\sum_{a=0}^{\lfloor (k+1)/2 \rfloor} \chi (a,0,k+1) \\
&\qquad = \delta_k (q+1) \chi(\lfloor k/2 \rfloor+1,0,k+1)\\ 
&\qquad\quad-q \chi(0,0,k+1) -q \delta_k \chi(\lfloor (k+1)/2\rfloor,0,k+1)+
qC_{k+1}.
\end{align*}
And:
\begin{align*}
&\sum_{a=0}^{\lfloor k/2 \rfloor}n_4(a,0,k) \chi (a+1,1,k+1)\\
&\qquad = \sum_{a=0}^{\lfloor k/2 \rfloor}n_4(a,0,k) \chi (a,0,k-1)\\
&\qquad = \sum_{a=0}^{\lfloor (k-1)/2 \rfloor}n_4(a,0,k) \chi (a,0,k-1)\\
&\qquad = \delta_k q^2 \chi(\lfloor (k-1)/2\rfloor ,0,k-1) + 
q^3 \sum_{a=0}^{\lfloor (k-1)/2 \rfloor} \chi (a,0,k-1)\\
&\qquad = \delta_k q^2 \chi(\lfloor (k-1)/2\rfloor ,0,k-1) + 
q^3 C_{k-1}.
\end{align*}
And
\begin{align*}
&\varepsilon \sum_{a=0}^{\lfloor k/2 \rfloor}n_5(a,0,k)  \chi(a,0,k)\\
&\qquad = \varepsilon n_5(0,0,k) \chi(0,0,k) 
+ \varepsilon \sum_{a=1}^{\lfloor k/2 \rfloor}n_5(a,0,k)  \chi(a,0,k)\\
&\qquad = \varepsilon (q-1) \chi(0,0,k) 
+ \varepsilon (q^2-1) \sum_{a=1}^{\lfloor k/2 \rfloor}  \chi(a,0,k)\\
&\qquad = \varepsilon (q-q^2) \chi(0,0,k) 
+ \varepsilon (q^2-1) \sum_{a=0}^{\lfloor k/2 \rfloor}  \chi(a,0,k)\\
&\qquad = \varepsilon (q-q^2) \chi(0,0,k) 
+ \varepsilon (q^2-1) C_k.
\end{align*}
Adding, we find that:
\begin{align*}
\chi(X) C_k
&= (q+1)C_{k+1} +\varepsilon (q^2-1) C_k+ (q^2+q^3)C_{k-1}\\
& \quad + q \chi(0,1,k+1) -q a_{k+1} + \varepsilon (q-q^2) a_k.
\end{align*}
By   \eqref{akeq1} we have
\begin{align*}
q \chi(0,1,k+1) &= -2^{-1}\left((q-1) \varepsilon- \chi(X)  \right) a_k 
-2^{-1}a_{k+1} 
-2^{-1}q^3 a_{k-1}.
\end{align*}
Therefore, 
\begin{align*}
&2C_{k+1} +(q+1)^{-1}\left( \varepsilon  (q^2-1)-\chi(X)\right) 2C_k+ 
(q+1)^{-1}(q^2+q^3)2C_{k-1}\\
&\qquad =
(q+1)^{-1} (2q+1)a_{k+1}  -(q+1)^{-1}\left( 2\varepsilon 
(q-q^2)-\left((q-1) 
\varepsilon- \chi(X)
\right)\right) a_k\\
&\qquad\quad +(q+1)^{-1}q^3 a_{k-1}.
\end{align*}
Using $2C_j = B_j + a_j$ for $j \in \Z_{\geq 0}$ we now obtain 
\eqref{Bkeq2}. 
\end{proof}

\begin{lemma}
\label{cjsumlemma}
Let $\chi \in \Hom(\mathcal{H},\C)_1$ or $\chi \in \Hom(\mathcal{H},\C)_2$. Assume that $\chi(1,1,2)=1$
and define 
$\varepsilon = \chi(V)$ (so that $\varepsilon=\pm 1$ by Lemma \ref{alphabetainvlemma}).
For $k \in \Z_{\geq 0}$ define $a_k = a_k(\chi)$ and
$B_k=B_k(\chi)$ as in \eqref{akeq2} and \eqref{Bkeq1}. Define
\begin{equation}
\label{cjsum1000}
\left\{
\begin{array}{l}
c_0=\varepsilon q^4,\\
c_1 = -q-q^2+q^3-\varepsilon q \chi(X),\\
c_2=-\varepsilon -\chi(X),\\
c_3=1
\end{array}
\right\}
\qquad
\text{if $\chi \in \Hom(\mathcal{H},C)_1$}.
\end{equation}
and
\begin{equation}
\label{cjsum1001}
\left\{
\begin{array}{l}
c_0=-\varepsilon q^5,\\
c_1 = q^2-2q^4+\varepsilon q^2 \chi(X),\\
c_2=-\varepsilon+\varepsilon q +\varepsilon q^2-\chi(X),\\
c_3=1
\end{array}
\right\}
\qquad
\text{if $\chi \in \Hom(\mathcal{H},C)_2$}.
\end{equation}
and
\begin{equation}
\label{cjeq1}
A_k  = \sum_{j=0}^3 c_j \chi( T(q^{k+j})) \qquad \text{for $k \in 
\Z_{\geq 0}$}.
\end{equation}
Then
\begin{gather}
A_0=A_1=A_2=0, \label{cjeq10}\\
\sum_{j=0}^3 c_j a_{k+j} = 0 \qquad \text{for $k \geq 1$}, \label{ceq11}\\
\sum_{j=0}^3 c_j B_{k+j} = 0 \qquad \text{for $k \geq 2$}. \label{cjeq2}
\end{gather}
\end{lemma}
\begin{proof}
To prove that $A_0=A_1=A_2=0$ we first need to calculate $\chi(T(q^k))$ for 
$0\leq k \leq 5$.  
Let $k \in \Z_{\geq 0}$, and let $t = \mathrm{rank}_\Z(\mathcal{H}_k)$. Let
$C_1,\dots,C_t$ be a listing of the $C \in K \backslash \Delta / K$ such that 
$C \in \mathcal{H}_k$, and let $C_1',\dots,C_t'$ be a list of the elements of 
$\mathcal{B}_k$ from Theorem~\ref{genreltheorem}. Then $C_1,\dots,C_t$ and
$C_1,\dots,C_t'$ are ordered $\Z$-bases for $\mathcal{H}_k$. Let $M_k = 
(m_{ij}) \in \GL(t,\Z)$ be such that $C_i = \sum_{j=1}^t m_{ij}C_j'$ for $i \in 
\{1,\dots,t\}$. Then 
$$
\chi(T(q^k)) = \sum_{i,j=1}^t m_{ij} \chi(C_j'). 
$$
Using now  Theorem~\ref{degree1multiplicationthm},
Theorem~\ref{degree2multiplicationthm}, and a computer algebra system it is
straightforward to determine $M_k$ and calculate $\chi(T(q^k))$ for $0 \leq k 
\leq 5$. 
For example, we have
$$
\begin{bsmallmatrix}
&1&&& \vphantom{T(1,0,2)}\\
-(q+1)&-(q+1)&-q(q+1)^2&-q+1&1\vphantom{T(1,0,2)} \\
&&&1&\vphantom{T(1,0,2)}\\
&&1&&\vphantom{T(1,0,2)}\\
1&&&&\vphantom{T(1,0,2)}
\end{bsmallmatrix}
\begin{bsmallmatrix}
Y_1\vphantom{T(1,0,2)}\\
Y_2\vphantom{T(1,0,2)} \\
V^2\vphantom{T(1,0,2)} \\
VX \vphantom{T(1,0,2)}\\
X^2\vphantom{T(1,0,2)}
\end{bsmallmatrix}
=
\begin{bsmallmatrix}
T(1,0,2)\\
T(0,0,2)\\
VT(0,0,1)\\
T(1,1,2)\\
T(0,1,2)
\end{bsmallmatrix}
$$
so that, recalling that $\chi(V) = \varepsilon \in \{\pm 1\}$ and $\chi(Y_2) = 
\chi(Y_1)$, 
\begin{equation}
\label{cjeq101}
\chi (T(q^2)) = -2q \chi (Y_1) + 1 - q(q+1)^2+ 
(2-q) \varepsilon \chi (X) + \chi (X)^2.
\end{equation}
With these formulas for $\chi(T(q^k))$ for $0 \leq k \leq 5$ we then find that $A_0=A_1=A_2 =0$. In fact, for
$k \in \{0,1,2\}$, $A_k = (1-\varepsilon)(1+\varepsilon) p(\varepsilon, 
\chi(X))$ where $p$ is polynomial in two variables over $\Z$; since $\varepsilon 
=\pm 1$, this is zero. 

Next, by \eqref{akeq3} we have
$$
\sum_{j=0}^3 c_j a_{k+4+j} 
= A \sum_{j=0}^3 c_j a_{k+3+j} + B \sum_{j=0}^3 c_j a_{k+2+j} 
+ C \sum_{j=0}^3 c_j a_{k+1+j} + D \sum_{j=0}^3 c_j a_{k+j} 
$$
for $k \geq 1$ and $A, B, C,$ and $D$ as in \eqref{akeq35}. To prove \eqref{ceq11} it thus suffices to prove that
\begin{equation}
\label{cjeq13}
\sum_{j=0}^3 c_j a_{4+j} = \sum_{j=0}^3 c_j a_{3+j} 
=\sum_{j=0}^3 c_j a_{2+j} = \sum_{j=0}^3 c_j a_{1+j} =0.
\end{equation}
Again using  Theorem~\ref{degree1multiplicationthm},
Theorem~\ref{degree2multiplicationthm}, and a computer algebra system, one may
calculate $a_\ell= \chi(0,0,\ell)$ for $0 \leq \ell \leq 7$ and then verify 
\eqref{cjeq13}. 

For the last assertion of the lemma we first note that by \eqref{Bkeq2} and \eqref{ceq11},
\begin{align*}
\sum_{j=0}^3 c_j B_{k+2+j}
&=
(q+1)^{-1}\left( \varepsilon (1-q^2)+\chi(X) 
\right) \sum_{j=0}^3 c_j B_{k+1+j} \\
&\quad -(q+1)^{-1}(q^2+q^3) \sum_{j=0}^3 c_j 
B_{k+j}
\end{align*}
for $k \geq 2$. It follows that to prove \eqref{cjeq2} it suffices to prove
that 
$$
\sum_{j=0}^3 c_j B_{3+j} = \sum_{j=0}^3 c_j B_{2+j} =0.
$$
This once again follows by a calculation using  Theorem~\ref{degree1multiplicationthm},
Theorem~\ref{degree2multiplicationthm}, and a computer algebra system.
\end{proof}

We can now prove the main result of this section.
 
\begin{theorem}
\label{charrationalitytheorem}
Let $\chi \in \Hom(\mathcal{H},\C)_1$ or $\chi \in \Hom(\mathcal{H},\C)_2$.
Assume that $\chi(1,1,2)=1$ and define $\varepsilon = \chi(V)$
(so that $\varepsilon=\pm 1$ by Lemma \ref{alphabetainvlemma}).
Let $t$ be an indeterminate. Then in $\C[[t]]$ we have
\begin{equation}
\sum_{k=0}^\infty \chi (T(q^k)) t^k = \dfrac{P_\chi(t)}{Q_\chi(t)}
\end{equation}
where
\begin{align*}
P_\chi(t)& =
\begin{cases}
1-q^2 t^2 & \text{if $\chi \in \Hom(\mathcal{H},\C)_1$},\\
1 + \varepsilon (q  + q^2 )t + q^3 t^2 &\text{if $\chi \in \Hom(\mathcal{H},\C)_2$,}
\end{cases} \\
Q_\chi(t)&=
\begin{cases}
1 -(\varepsilon + \chi(X)) t + (q^3-q^2-q-\varepsilon q \chi(X) )t + \varepsilon q^4 t^3&\\
\qquad \text{if $\chi \in \Hom(\mathcal{H},\C)_1$}, & \\
1-\left(\varepsilon (1-q-q^2) + \chi (X) \right)t +\left(q^2-2q^4 +\varepsilon q^2 \chi(X) \right) t^2 -\varepsilon q^5 t^3 & \\
\qquad \text{if $\chi \in \Hom(\mathcal{H},\C)_2$}. &
\end{cases}
\end{align*}
\end{theorem}
\begin{proof}
Define $c_0,c_1,c_2,$ and $c_3$ as in \eqref{cjsum1000} and \eqref{cjsum1001}. Then $Q_\chi(t) = \sum_{j=0}^3 c_{3-i} t^i$. A calculation
shows that
\begin{align*}
Q_\chi(t) \sum_{k=0}^\infty \chi (T(q^k)) t^k
=R(t)+ \sum_{\ell=3}^\infty A_{\ell-3} t^\ell
\end{align*}
where we define
\begin{align*}
R(t)& = c_3 +\left(c_3 \chi(T(q)) + c_2\right) t +
\left(c_3 \chi (T(q^2)) + c_2 \chi (T(q)) +c_1\right) t^2,\\
A_k & = \sum_{j=0}^3 c_j \chi( T(q^{k+j})) \qquad \text{for $k \in 
\Z_{\geq 0}$}.
\end{align*}
To prove the theorem it will suffice to prove that $R(t) =P_\chi(t)$ and that
$A_k =0$ for $k \geq 0$.  A calculation using \eqref{cjeq101}, $\chi(T(q)) =
\varepsilon + \chi (X)$,  \eqref{cjsum1000},  and \eqref{cjsum1001} verifies
that $R(t) = P_\chi(t)$.
Next, for $k \in \Z_{\geq 0}$ define
$$
T_0(q^k) = \sum_{(a,b,k) \in S_\Delta} T(a,b,k).
$$
Assume that $k \in \Z_{\geq 0}$ and $k \geq 2$.
If $(a,b,k) \in S_\Delta$ and $a,b >0$, then $T(a,b,k) = T(1,1,2)
T(a-1,b-1,k-2)$ so that 
$$
\chi (a,b,k) = \chi (T(1,1,2)) \chi ( 
T(a-1,b-1,k-2)) = \chi (a-1,b-1,k-2).
$$
Hence, 
\begin{align}
\chi (T_0(q^k)) 
& = \sum_{(a,b,k) \in S_\Delta} \chi(a,b,k)\nonumber\\
& = \sum_{\substack{(a,b,k) \in S_\Delta, \\ a,b>0}} \chi(a,b,k) +
\sum_{\substack{(a,b,k) \in S_\Delta, \\ ab=0}} \chi(a,b,k)\nonumber \\
& = \sum_{\substack{(a,b,k) \in S_\Delta, \\ a,b>0}} \chi(a-1,b-1,k-2) +
\sum_{\substack{(a,b,k) \in S_\Delta, \\ ab=0}} \chi(a,b,k)\nonumber \\
& = \sum_{\substack{(a,b,k-2) \in S_\Delta }} \chi(a,b,k-2) +
\sum_{\substack{(a,b,k) \in S_\Delta, \\ ab=0}} \chi(a,b,k)\nonumber \\
& = \chi(T_0(q^{k-2}) )+ B_k \label{creq1}
\end{align}
where $B_k$ is defined as in \eqref{Bkeq1}. 
Now by Lemma~\ref{BTdecomplemma}
\begin{equation}
\label{creq3}
T(q^k) = VT_0(q^{k-1}) + T_0(q^k).
\end{equation}
Applying $\chi$ to \eqref{creq3} and using $\chi(V) = \varepsilon$ and 
\eqref{creq1} we obtain:
\begin{align}
\chi(T(q^k)) &= \varepsilon \chi (T_0(q^{k-1}) )+ \chi( T_0(q^k))\nonumber \\
& =  \varepsilon \chi (T_0(q^{k-1}) ) + \chi(T_0(q^{k-2}) )+ B_k.\label{creq4}
\end{align}
Assume further that $k \geq 3$. Starting from \eqref{creq4} and using 
\eqref{creq3} again, we have:
\begin{align*}
\chi(T(q^k)) 
& =  \varepsilon \chi (T_0(q^{k-1}) ) + \chi(T_0(q^{k-2}) )+ B_k\\
& =  \varepsilon \chi \left( T(q^{k-1}) -V T_0(q^{k-2} )\right) + 
\chi(T_0(q^{k-2}) )+ B_k\\
& =  \varepsilon \left( \chi ( T(q^{k-1}))  -\varepsilon \chi( T_0(q^{k-2} 
))\right) + 
\chi(T_0(q^{k-2}) )+ B_k\\
& =  \varepsilon  \chi  (T(q^{k-1}))  -\varepsilon^2 \chi( T_0(q^{k-2} 
)) + 
\chi(T_0(q^{k-2}) )+ B_k\\
& =  \varepsilon  \chi  (T(q^{k-1}))  -\chi( T_0(q^{k-2} 
)) + 
\chi(T_0(q^{k-2}) )+ B_k\\
& =  \varepsilon  \chi  (T(q^{k-1})) + B_k.
\end{align*}
Note that we used $\varepsilon^2=1$. Thus
$$
\chi(T(q^{k+1})) = \varepsilon \chi(T(q^{k})) + B_{k+1} \qquad \text{for $k 
\geq 2$.}
$$
It follows that for $k \geq 2$,
\begin{align*}
\sum_{j=0}^3 c_j \chi(T(q^{k+1+j})) & = \sum_{j=0}^3 c_j \chi(T(q^{k+j})) + 
\sum_{j=0}^3 c_j B_{k+1+j}\\
A_{k+1} & = A_k + \sum_{j=0}^3 c_j B_{k+1+j}.
\end{align*}
Using this equation, \eqref{cjeq10}, and \eqref{cjeq2} we see that $A_k =0$ for 
$k \geq 0$, which completes the proof.
\end{proof}

We make two remarks concerning Theorem \ref{charrationalitytheorem}.
First, we note that the expression for $Q_\chi(t)$ in Theorem \ref{charrationalitytheorem}
when $\chi \in \Hom(\mathcal{H},\C)_1$ is as expected. To explain this, assume
that $(\pi,V)$ is an irreducible, admissible representation of $\GSp(4,F)$ with
trivial central character and paramodular level $N_\pi=1$, so that the vector
space $V(1)$ in \eqref{level1eq} is one-dimensional. Let $\varphi_\pi$ be the
$L$-parameter of $\pi$ as in \cite{RS}. Then
by Theorem 7.5.9 of \cite{RS} the $L$-factor of $\varphi_\pi$ is
$L(s,\varphi_\pi) = D(q^{-s})$ where
\begin{equation}
\label{Ddefeq}
D(t)
=
1 - q^{-3/2} \left(\chi(X) + \chi(V) \right) t +
\left( q^{-2} \chi(Y_1) +1 \right) t^2
+ \chi(V) q^{-1/2} t^3.
\end{equation}
Using  the definition of $\Hom(\mathcal{H},\C)_1$ and Theorem \ref{charrationalitytheorem}  we have
\begin{equation}
\label{PDeq}
Q_\chi (t ) = D(q^{3/2}t).
\end{equation}
This may be regarded as an expected
equality because if $\pi$ is unramified, then the analogue of \eqref{PDeq}
holds. See Theorem 2 of \cite{S1} and Theorem 7.5.9 of \cite{RS}.

Second, we remark that the assertion of Theorem 2 of \cite{S1} is an identity
in the unramified Hecke algebra of $\GSp(4,F)$, while Theorem \ref{charrationalitytheorem}
provides an identity  after the application of $\chi$ to elements of $\mathcal{H}$;
consequently, one might speculate that an identity in $\mathcal{H}$ is behind
Theorem \ref{charrationalitytheorem}.
We have not eliminated this possibility, but we can show that the natural
 guess does not hold.
In light of the previous paragraph,
one would expect that $\sum_{k=0}^\infty T(q^k) t^k$
is equal to
\begin{equation}
\label{guesseq}
\left(1 - q^2 V^2 t^2\right) \left(
1 -  \left(X + V \right) t +
\left( q Y_1 + q^3 V^2 \right) t^2
+ V^3 q^4 t^3 \right)^{-1}.
\end{equation}
However, Theorem \ref{charrationalitytheorem} rules this out. For assume that
$\sum_{k=0}^\infty T(q^k) t^k$ is equal to \eqref{guesseq}, and let
$\chi = \chi_{\mathrm{index}}$. Then by our assumption, \eqref{indexeq1},
\eqref{indexeq2}, and \eqref{indexeq3},
$\sum_{k=0}^\infty \chi(T(q^k)) t^k$
is equal to
\begin{equation}
\label{guessindeq}
\left(1 - q^2 t^2\right) \left(
1 -  \left(1 + q + 2q^2 + q^3\right) t +
\left( q^3+q^4+ q^5 \right) t^2
+  q^4 t^3 \right)^{-1}.
\end{equation}
On the other hand, by Theorem \ref{charrationalitytheorem} we find that
$\sum_{k=0}^\infty \chi(T(q^k)) t^k$ is equal to
\begin{equation}
\label{chiindeq}
\left(1+(q+q^2)t+q^3 t^2\right)
\left( 1 - \left( 1+q^2+q^3 \right) t
+\left( q^2+q^3+q^5 \right) t^2 - q^5 t^3
\right)^{-1}.
\end{equation}
Since \eqref{guessindeq} is not equal to \eqref{chiindeq} we have a contradiction.

\begin{corollary}
If $k \in \Z_{\geq 0}$, then
\begin{equation}
\chi_{\mathrm{index}} (T(q^k))
=
\dfrac{1+q^2 -2q^{2k+1}(1+q+q^2)+q^{3k+1}(1+q+q^2+q^3)}{(q-1)^2(1+q+q^2)}.
\end{equation}
\end{corollary}
\begin{proof}
This follows from \eqref{indexeq1}, \eqref{indexeq2}, \eqref{indexeq3}, and Theorem~\ref{charrationalitytheorem} via standard techniques.
\end{proof}
The first few values of $\chi_{\mathrm{index}} (T(q^k))$ are
$$
\begin{array}{ll}
\toprule
k & \chi_{\mathrm{index}} (T(q^k)) \\
\midrule
0 & 1 \\
1 & 1 + q + 2 q^2 + q^3 \\
2 & 1 + q + 2 q^2 + 3 q^3 + 3 q^4 + 2 q^5 + q^6\\
3 & 1 + q + 2 q^2 + 3 q^3 + 3 q^4 + 4 q^5 + 5 q^6 + 3 q^7 + 2 q^8 + q^9\\
\bottomrule
\end{array}
$$

\section{The center}

In this section we determine the center $Z(\mathcal{H})$ of $\mathcal{H}$.
To begin, we note that $Z(\mathcal{H})$ is a graded algebra.
For $k \in \Z_{\geq 0}$ define
\begin{equation}
\label{Akdefeq}
\mathcal{A}_k = Z (\mathcal{H}) \cap \mathcal{H}_k.
\end{equation}
Then  $\mathcal{A}_k \mathcal{A}_j \subset \mathcal{A}_{k+j}$
for $k,j \in \Z_{\geq 0}$ and also
\begin{equation}
\label{Zgradedeq}
Z(\mathcal{H}) = \bigoplus_{k \in \Z_{\geq 0}} \mathcal{A}_k.
\end{equation}

\begin{lemma}
\label{Vcancellationlemma}
Let $R \in \mathcal{H}$ and $j \in \Z_{\geq 0}$. If $V^j R =0$ or $R V^j=0$, then $R=0$.
\end{lemma}
\begin{proof}
We may assume that $R \in \mathcal{H}_k$ for some $k \in \Z_{\geq 0}$. If $V^j R =0$, then $R=0$
by Theorem~\ref{genreltheorem}. If $RV^j=0$, then $V^j \beta(R) =0$ where $\beta: \mathcal{H} \to \mathcal{H}$
is the anti-automorphism of $\mathcal{H}$ defined in section~\ref{structuresec}.
Hence, $\beta(R)=0$, so that $R=0$.
\end{proof}

\begin{lemma}
\label{centerform}
If $Z\in Z(\mathcal{H})$, then $Z$ is a $\Z$-linear combination of elements of the form $V^aX^b(Y_1^c+Y_2^c)$, where $a,b,c\in\mathbb{Z}_{\ge 0}$.
\end{lemma}
\begin{proof}
We may assume that $Z \in \mathcal{H}_k$ for some $k \in \Z_{\geq 0}$.
By Theorem~\ref{genreltheorem} we may write
$$
Z=\underset{\substack{a,b,c_1\in\mathbb{Z}_{\ge 0},\\ a+b+2c_1=k}}{\sum}n_1(a,b,c_1)V^aX^bY_1^{c_1}+\underset{\substack{a,b,c_2\in\mathbb{Z}_{\ge 0},\\ a+b+2c_2=k}}{\sum}n_2(a,b,c_2)V^aX^bY_2^{c_2}
$$
where $n_1(a,b,c),n_2(a,b,c) \in \Z$.
As $Z\in Z(\mathcal{H})$, we have that $ZV=VZ$, and hence by \eqref{r1eq}, \eqref{r2eq}, and \eqref{r3eq}
$$
\underset{\substack{a,b,c_1\in\mathbb{Z}_{\ge 0},\\ a+b+2c_1=k}}{\sum}n_1(a,b,c_1)V^{a+1}X^bY_2^{c_1}+\underset{\substack{a,b,c_2\in\mathbb{Z}_{\ge 0},\\ a+b+2c_2=k}}{\sum}n_2(a,b,c_2)V^{a+1}X^bY_1^{c_2}
$$
is equal to
$$
\underset{\substack{a,b,c_1\in\mathbb{Z}_{\ge 0},\\ a+b+2c_1=k}}{\sum}n_1(a,b,c_1)V^{a+1}X^bY_1^{c_1}+\underset{\substack{a,b,c_2\in\mathbb{Z}_{\ge 0},\\ a+b+2c_2=k}}{\sum}n_2(a,b,c_2)V^{a+1}X^bY_2^{c_2}.
$$
The assertion of the lemma follows now from Theorem~\ref{genreltheorem}.
\end{proof}

\begin{lemma}
\label{Qresult}
Let $Q\in \mathcal{H}$ be such that
\begin{equation}
\label{Qresulteq}
Y_1Q=QY_2,\qquad Y_2Q=QY_1,\qquad VQ=QV, \qquad XQ=QX.
\end{equation}
Then $Q=VZ$ for some $Z\in Z(\mathcal{H})$.
\end{lemma}
\begin{proof}
Let $n\in\mathbb{Z}_{\ge 0}$ be such that $Q=\sum_{k=0}^n Q_k$ where $Q_k\in \mathcal{H}_k$ for $0 \leq k \leq n$.
Note that each $Q_k$ has same properties as $Q$, and so we may assume that $Q\in \mathcal{H}_k$ for some $k\in\mathbb{Z}_{\ge 0}$.
We may also assume that $k>0$.
By Theorem~\ref{genreltheorem} we have
$$
Q=\underset{\substack{a,b,c\in\mathbb{Z}_{\ge 0},\\ a+b+2c=k}}{\sum}n_1(a,b,c)V^aX^bY_1^{c}
+\underset{\substack{a,b,c\in\mathbb{Z}_{\ge 0},\\ a+b+2c=k}}{\sum}n_2(a,b,c)V^aX^bY_2^{c}
$$
for some $n_1(a,b,c),n_2(a,b,c) \in \Z$.
We claim that $n_1(0,b,c)=n_2(0,b,c)=0$ for $b,c \in \Z_{\geq 0}$ with $b+2c =k$.
To prove this,  let $\chi\in \text{Hom}(\mathcal{H},\mathbb{C})_3$ with $\chi(V)=0$ and $\chi(Y_2)=0$
(see Lemma~\ref{charactertypelemma}). Then
$$
\chi(Q)=\underset{\substack{b,c\in\mathbb{Z}_{\ge 0},\\ b+2c=k}}{\sum}n_1(0,b,c)\chi(X)^b\chi(Y_1)^{c}.
$$
Since $Y_1Q=QY_2$ we also have
$$
\underset{\substack{b,c\in\mathbb{Z}_{\ge 0},\\ b+2c=k}}{\sum}n_1(0,b,c)\chi(X)^b\chi(Y_1)^{c+1}=\chi(Y_1Q)=\chi(QY_2)=0.
$$
We conclude that for indeterminates $x$ and $y$
$$
0=\underset{\substack{b,c\in\mathbb{Z}_{\ge 0},\\ b+2c=k}}{\sum}n_1(0,b,c)x^b y^{c+1}.
$$
Hence, $n_1(0,b,c)=0$ for $b,c\in\mathbb{Z}_{\ge 0}$ with $b+2c=k$. Similarly,
$n_2(0,b,c)=0$ for $b,c\in\mathbb{Z}_{\ge 0}$ with $b+2c=k$, thus proving our claim.
We now have that $Q=VZ$ for some $Z\in \mathcal{H}_{k-1}$. Using the properties of $Q$ and
\eqref{r1eq}, \eqref{r2eq}, and \eqref{r3eq} we find that
\begin{gather*}
V^2Z=VQ=QV=VZV,\\
VXZ=XVZ=XQ=QX=VZX,\\
VY_1Z=Y_2VZ=Y_2Q=QY_1=VZY_1,\\
VY_2Z=Y_1VZ=Y_1Q=QY_2=VZY_2.
\end{gather*}
Using Lemma~\ref{Vcancellationlemma} to cancel
$V$ from the above equations we find that $Z$ commutes with $V$, $X$, $Y_1$, and $Y_2$ and
thus lies in $Z(\mathcal{H})$ by Theorem~\ref{gentheorem}.
\end{proof}
\begin{lemma}
\label{elements}
Let
\begin{equation}
\label{elementslemmaeq1}
Z_1=X-(q^2-1)V,\quad Z_2=(q-1)VX-(Y_1+Y_2),\qquad Z_3=V^2.
\end{equation}
Then $Z_1,Z_2,Z_3 \in Z(\mathcal{H})$.
\end{lemma}
\begin{proof}
In this proof we will repeatedly use Theorem~\ref{gentheorem} and \eqref{r1eq}, \eqref{r2eq}, and \eqref{r3eq}.
It clear that $Z_3=V^2\in Z(\mathcal{H})$ and that $Z_1$ commutes with $X$ and $V$. Now
\begin{align*}
Y_1Z_1=&Y_1X-(q^2-1)VY_2\\
=&XY_1-(q^2-1)VY_1+(q^2-1)VY_2-(q^2-1)VY_2\\
=&Z_1Y_1
\end{align*}
and
\begin{align*}
Y_2Z_1=&Y_2X-(q^2-1)VY_1\\
=&XY_2+(q^2-1)VY_1-(q^2-1)VY_2-(q^2-1)VY_1\\
=&Z_1Y_2.
\end{align*}
Hence, $Z_1\in Z(\mathcal{H})$. For $Z_2$, it is clear that $Z_2$ commutes with $V$. Now
\begin{align*}
XZ_2&=(q-1)VX^2-(XY_1+XY_2)\\
&=(q-1)VX^2-Y_1X+(1-q^2)VY_1 -(1-q^2)VY_2\\
&\quad -Y_2X-(1-q^2)VY_1+(1-q^2)VY_1\\
&=(q-1)VX^2-(Y_1X+Y_2X)\\
&=Z_2X,
\end{align*}
\begin{align*}
Y_1Z_2&=(q-1)V(Y_2X)-(Y_1^2+Y_1Y_2)\\
&=(q-1)VXY_2+(q-1)(q^2-1)V^2Y_1-(q-1)(q^2-1)V^2Y_2\\
&\quad-Y_1^2-Y_2Y_1 -(q-1)^2 (q+1) V^2Y_1 +(q-1)^2(q+1)V^2Y_2\\
&\quad+(q-1)VXY_1-(q-1)VXY_2\\
=&(q-1)VXY_1-(Y_1^2+Y_2Y_1)\\
=&Z_2Y_1,
\end{align*}
and
\begin{align*}
Y_2Z_2&=(q-1)V(Y_1X)-(Y_2Y_1+Y_2^2)\\
&=(q-1)VXY_1-(q-1)(q^2-1)V^2Y_1+(q-1)(q^2-1)V^2Y_2\\
&\quad-Y_1Y_2-Y_2^2 +(q-1)^2 (q+1) V^2Y_1 -(q-1)^2(q+1)V^2Y_2\\
&\quad-(q-1)VXY_1+(q-1)VXY_2\\
&=(q-1)VXY_2-(Y_1Y_2+Y_2^2)\\
&=Z_2Y_2.
\end{align*}
It follows that $Z_2\in Z(\mathcal{H})$, proving the lemma.
\end{proof}

\begin{theorem}
\label{centertheorem}
Define $Z_1$, $Z_2$, and $Z_3$ as in \eqref{elementslemmaeq1}.
Then $Z(\mathcal{H})$ is generated over $\mathbb{Z}$ by $Z_1, Z_2$ and $Z_3$, and
$Z_1$, $Z_2$, and $Z_3$ are algebraically independent over $\Z$.
\end{theorem}
\begin{proof}
We have $Z_1,Z_2,Z_3\in Z(\mathcal{H})$ by Lemma~\ref{elements}.
As mentioned at the beginning of this section, $Z(\mathcal{H}) = \oplus_{k \in \Z_{\geq 0}} \mathcal{A}_k$.
Let $k \in \Z_{\geq 0}$. To prove the first assertion of the theorem it will suffice to
prove the following claim: if $Z \in \mathcal{A}_k$,
then $Z$ is a polynomial in $Z_1,Z_2$, and $Z_3$. We will prove this claim by induction on $k$.
Direct computations prove that the claim holds if $k=0$ or $k=1$. Let $k \in \Z$ with $k \geq 2$,
and assume that the claim holds for $j \in \Z$ with $0 \leq j <k$; we will prove that the claim
holds for $k$. Let $Z \in \mathcal{A}_k$.
Let the $J$ be the two-sided ideal of $\mathcal{H}$ generated by $V$.
By Lemma~\ref{centerform} we have that
$$
Z=\underset{\substack{a,b,c\in\mathbb{Z}_{\ge 0},\\ a+b+2c=k}}{\sum}n(a,b,c)V^{a}X^b(Y_1^{c}+Y_2^c)
$$
for some $n(a,b,c)\in\mathbb{Z}$.
By \eqref{r6eq} and \eqref{r7eq}, if $c\in\mathbb{Z}_{\ge 0}$, then
$(Y_1+Y_2)^c\equiv Y_1^c+Y_2^c\Mod{J}$.
Hence,
\begin{align*}Z\equiv&\underset{\substack{a,b,c\in\mathbb{Z}_{\ge 0},\\ a+b+2c=k}}{\sum}n(a,b,c)V^{a}X^b(Y_1+Y_2)^c\Mod{J}
\\\equiv&\underset{\substack{a,b,c\in\mathbb{Z}_{\ge 0},\\ a+b+2c=k}}{\sum}n(a,b,c)V^{a}X^b((q-1)VX-Z_2)^c\Mod{J}
\\\equiv&\underset{\substack{b,c\in\mathbb{Z}_{\ge 0},\\ b+2c=k}}{\sum}n(0,b,c)X^b((q-1)VX-Z_2)^c\Mod{J}
\\\equiv&\underset{\substack{b,c\in\mathbb{Z}_{\ge 0},\\ b+2c=k}}{\sum}(-1)^cn(0,b,c)X^bZ_2^c\Mod{J}
\\\equiv&\underset{\substack{b,c\in\mathbb{Z}_{\ge 0},\\ b+2c=k}}{\sum}(-1)^cn(0,b,c)(Z_1+(q^2-1)V)^bZ_2^c\Mod{J}
\\\equiv&\underset{\substack{b,c\in\mathbb{Z}_{\ge 0},\\ b+2c=k}}{\sum}(-1)^cn(0,b,c)Z_1^bZ_2^c\Mod{J}.
\end{align*}
It follows that there is a polynomial $p(x,y)\in\mathbb{Z}[x,y]$ and  $Q\in\mathcal{H}_{k-1}$ such that
$
Z=p(Z_1,Z_2)+VQ.
$
Since $VQ=Z-p(Z_1,Z_2)$ we have $VQ\in Z(\mathcal{H})$.
Using \eqref{r1eq}, \eqref{r2eq}, \eqref{r3eq}, and $VQ\in Z(\mathcal{H})$ we find that
$V^2Q=VQV$,  $VXQ=XVQ=VQX$, $VY_1Q=Y_2VQ=VQY_2$, and $VY_2Q=Y_1VQ=VQY_1$.
Using Lemma~\ref{Vcancellationlemma} to cancel $V$ from the just mentioned equations, we find that \eqref{Qresulteq} holds.
Hence there exists  $Z'\in Z(\mathcal{H}) \cap \mathcal{H}_{k-2}$ such that $Q=VZ'$ by Lemma~\ref{Qresult}.
By the induction hypothesis
there exists a polynomial $q(x,y,z)\in\mathbb{Z}[x,y,z]$ such that $Z'=q(Z_1,Z_2,Z_3)$. Our claim follows.

Let $k \in \Z_{\geq 0}$. To prove that $Z_1$, $Z_2$, and $Z_3$ are algebraically independent over $\Z$
it will suffice to prove the following claim: if
$$
0 = \sum_{\substack{a,b,c \in \Z_{\geq 0},\\ a+2(b+c)=k}} n(a,b,c) Z_1^a Z_2^b Z_3^c
$$
for some $n(a,b,c) \in \Z$, then $n(a,b,c)=0$ for all $a,b,c \in \Z_{\geq 0}$ with $a+2(b+c)=k$.
This claim is clear if $k=0$ or $k=1$. Assume that $k \geq 2$ and that the claim holds for $j \in \Z_{\geq 0}$
with $0\leq j <k$; we will prove the claim for $k$. We first prove that $n(a,b,0)=0$ for $a,b \in \Z_{\geq 0}$
with $a+2b =k$. To see this,  let $\chi\in \text{Hom}(\mathcal{H},\mathbb{C})_3$ with $\chi(V)=0$ and $\chi(Y_2)=0$
(see Lemma~\ref{charactertypelemma}). Then since $Z_3=V^2$, we have
$$
0=\sum_{\substack{a,b \in \Z_{\geq 0},\\ a+2b=k}} n(a,b,0) \chi(Z_1)^a \chi(Z_2)^b
=\sum_{\substack{a,b\in \Z_{\geq 0},\\ a+2b=k}} n(a,b,0) \chi(-1)^b \chi(X)^a \chi(Y_1)^b.
$$
This implies that
$$
0 = \sum_{\substack{a,b\in \Z_{\geq 0},\\ a+2b=k}} n(a,b,0) \chi(-1)^b x^a y^b
$$
where $x$ and $y$ are indeterminates. Hence, $n(a,b,0)=0$ for $a,b \in \Z_{\geq 0}$ with $a+2b=0$.
It follows that
\begin{align*}
0& = \sum_{\substack{a,b,c \in \Z_{\geq 0},\\ a+2(b+c)=k,\\ c>0}} n(a,b,c) Z_1^a Z_2^b Z_3^c\\
& = \left(\sum_{\substack{a,b,c \in \Z_{\geq 0},\\ a+2(b+c)=k,\\ c>0}} n(a,b,c) Z_1^a Z_2^b Z_3^{c-1} \right) Z_3.
\end{align*}
Since $Z_3=V^2$, by Lemma~\ref{Vcancellationlemma} we have
$$
0=\sum_{\substack{a,b,c \in \Z_{\geq 0},\\ a+2(b+c)=k,\\ c>0}} n(a,b,c) Z_1^a Z_2^b Z_3^{c-1}.
$$
By the induction hypothesis, $n(a,b,c)=0$ for $a,b,c \in \Z_{\geq 0}$ with $a+2(b+c)=k$ and $c>0$. This proves the claim.
\end{proof}

Theorem \ref{centertheorem} and the Hilbert-Serre theorem can be used to calculate the Hilbert-Poincar\'e series of the
center $Z(\mathcal{H})$. (see Theorem~11.1 of \cite{AM} for the Hilbert-Serre theorem; see also Proposition 1 of Ch.~V, \S 5, no.~1 (p.~108) of \cite{B}). The following result directly calculates this series.

\begin{corollary}
\label{Zcorollary}
If $t$ is an indeterminate, then
\begin{equation}
\label{Zcoroeq}
\sum_{k=0}^\infty \mathrm{rank}_\Z (\mathcal{A}_k)\, t^k =
\dfrac{1}{(1-t)(1-t^2)^2}.
\end{equation}
\end{corollary}
\begin{proof}
Let $Z_1, Z_2$, and $Z_3$ be as in \eqref{elementslemmaeq1}. By Theorem \ref{centertheorem}
the elements $Z_1^a Z_2^b Z_3^c$ with $a,b,c \in \Z_{\geq 0}$ and $a+2(b+c) =k$ form a basis over $\Z$
for $\mathcal{A}_k$. This implies that
\begin{equation}
\label{Zocoroeq2}
\mathrm{rank}_\Z( \mathcal{A}_{k+2} ) =\mathrm{rank}_\Z(\mathcal{A}_k )+ \lfloor k/2 \rfloor +2
\end{equation}
for $k \in \Z_{\geq 0}$. A standard calculation now proves \eqref{Zcoroeq}.
\end{proof}

The first few ranks of $\mathcal{A}_k$ are
$$
\begin{array}{cllllllllll}
\toprule
k & 0 & 1 & 2 & 3 & 4 & 5 & 6 & 7 & 8 & 9  \\
\midrule
\text{rank}_\Z(\mathcal{A}_k) & 1 & 1 & 3 & 3 & 6 & 6 & 10 & 10 & 15 & 15\\
\bottomrule
\end{array}
$$

\bibliographystyle{plain}
\bibliography{PHA}

\end{document}